\documentclass[12pt,reqno]{amsart}
\usepackage{dsfont, amssymb,amsmath,amscd,latexsym, amsthm, amsxtra,amsfonts,enumerate}
\usepackage{mathrsfs}
\usepackage[all]{xy}
\usepackage[active]{srcltx}
\usepackage{graphicx}
\usepackage{bm}
\newtheorem{theorem}{Theorem}[section]

\newtheorem{Def}{Definition}[section]

\newtheorem{Lemma}{Lemma}[section]
\newtheorem{Assumption}{Assumption}[section]

\newcommand{\F}{\mathcal{F}}
\def\R{{\mathbb{R}}}

\newcommand{\E}{\mathbb{E}}

\numberwithin{equation}{section}

\begin{document}
\makeatletter
\def\@setauthors{%
\begingroup
\def\thanks{\protect\thanks@warning}%
\trivlist \centering\footnotesize \@topsep30\p@\relax
\advance\@topsep by -\baselineskip
\item\relax
\author@andify\authors
\def\\{\protect\linebreak}%
{\authors}%
\ifx\@empty\contribs \else ,\penalty-3 \space \@setcontribs
\@closetoccontribs \fi
\endtrivlist
\endgroup } \makeatother
 \baselineskip 17pt
\title[{{ \tiny Robust consumption-investment problem  }}]
 {{\tiny Robust consumption-investment problem Under CRRA and CARA utilities with time-varying confidence sets}}
\author[ Zongxia Liang and Ming Ma]
{Zongxia  Liang and Ming Ma$^{ \ *}$
\vskip 10pt\noindent
 Department of Mathematical Sciences, Tsinghua
University, Beijing 100084, China
  \footnote{$^*$ Corresponding author.\\ \noindent
zliang@math.tsinghua.edu.cn(Z.Liang),liangzongxia@mail.tsinghua.edu.cn(Z.Liang), \  mam14@mails.tsinghua.edu.cn(M.Ma) }}
\noindent
\begin{abstract}
We consider a robust consumption-investment problem under CRRA and CARA utilities.
The time-varying confidence sets are specified by $\Theta$, a correspondence from $[0,T]$ to the space of L\'{e}vy triplets, and describe priori information about drift, volatility and jump. Under each possible measure, the log-price processes of stocks are semimartingales and the triplet of their differential characteristics is a measurable selector from the correspondence $\Theta$
almost surely.
By proposing and studying the global kernel, an optimal policy and a worst-case measure are generated from a saddle point of the global kernel, and they also constitute a saddle point of the objective function.
\end{abstract}
\maketitle
\setcounter{equation}{0}
\section{Introduction}
We consider a robust Merton (consumption-investment) problem of the form
\begin{equation}\nonumber
\begin{aligned}
\sup_{\xi} \inf_{P} \bigg\{\E^P \big[\int_0^T U(C^{\xi}_t) dt +U(W^{\xi}_T)\big]\bigg\},
\end{aligned}
\end{equation}
in a continuous-time market with jumps. Given the investment period $[0,T]$, the objective function is the expectation of the sum of cumulative and terminal utilities, which is a function of consumption process $\{C^{\xi}_t\}_{t\in[0,T]}$ and terminal wealth $W^{\xi}_T$.
The policy $\xi$ uniquely determines $\{C^{\xi}_t\}_{t\in[0,T]}$ and $W^{\xi}_T$, that is, there is a mapping from policy to consumption process and terminal wealth. The utility function $U$ is strictly concave and increasing due to the risk aversion, such as CRRA and CARA utilities.
The supremum is taken over all admissible policies, and the infimum is over all possible measures for the log-price processes of the stocks. Rigorously, a policy $\xi$ is admissible if and only if it is predictable and the values of $\{C^{\xi}_t\}_{t\in[0,T]}$ and $W^{\xi}_T$ are in the domain of $U$. The model uncertainty is parameterized by $\Theta$, a correspondence from $[0,T]$ to the space of L\'{e}vy triplets $\R^{d}\times \mathbb{S}_+^d \times \mathcal{L}$. A measure $P$ is possible if and only if the log-price processes are semimartingales under $P$, and the triplet of their differential characteristics $(b^P, \Sigma^P, F^P)$ is a measurable selector from $\Theta$ almost surely.

\vskip 5pt
The consumption-investment problem is firstly studied by Merton (1969, 1971) with the Black-Scholes model. The Black-Scholes model assumes the log-price processes of stocks are described by a drifted Brownian motion, which facilities an explicit expression of optimal strategy.
As the research progresses, people try to weaken the assumptions for the dynamic of stocks and still hope to get an explicit or semi-explicit solution.
Kallsen (2000) and Nutz (2012) both use L\'{e}vy processes to characterize the log-price, so that the wealth process is an exponential L\'{e}vy process.
Stoikov and Zariphopoulou (2005), Kallsen and Muhle-Karbe (2010) and Benth et al. (2010) propose specific models for the volatility of log-price, such as Heston model, Carr model, and Barndorff-Nielsen-Shephard model.
Usually, a complex model can characterize the reality more accurately with fewer assumptions, but more parameters will be needed.
Especially when the market is nonstationary, the parameters could be time-varying and hardly estimated.
Therefore, it is reasonable to adopt a robust model to describe the market. In practice, people only need to estimate the confidence set of each parameter to establish a robust model. We herewith consider a robust model to weaken the assumption and fortunately obtain the optimal value function and policy in semi-explicit form.

\vskip 5pt
In the field of robust optimization, many papers focus on the investment (consumption-investment) problem.  Talay and Zheng (2002) suppose all admissible measures have a same equivalent martingale measure, which is a dominated robust model. Schied (2008) adds a stochastic factor into the model of volatility, which makes the market incomplete but is still a dominated model.
For nondominated robust models, Denis and Kervarec (2013) consider an investment problem under a bounded utility function, and Lin and Riedel (2014) solve a consumption-investment problem with the issue of non-equivalent multiple priors. Furthermore, Nutz (2016) and Neufeld and Nutz (2018) extend the model of stocks to the jump diffusion, for discrete-time and continuous-time investment problems, respectively.
In this paper, we are interested in the robust consumption-investment problem with jump diffusions, which has not been studied yet.
For each possible measure $P$, the differential characteristics of log-price processes constitute a triplet $(b^P, \Sigma^P, F^P)$, which is a measurable selector from the correspondence $\Theta$ almost surely. Thus $\Theta_t$ is a subset of all L\'{e}vy triplets and contains a priori information of drift, volatility and jump at time $t$. Nontrivially, $\{\Theta_t\}_{t\in[0,T]}$ contains time-varying confidence sets, which does not appear in previous studies and causes difficulties in demonstrating the measurability of worst-case differential characteristics.

\vskip 10pt
It is well known there are two commonly used methods to solve a stochastic optimization problem, namely the dynamic programming and the martingale method. The dynamic programming is firstly proposed by Merton (1971), which is from the idea of local optimization and describes the value function as a solution of a nonlinear PDE called Hamilton-Jacobi-Bellman (HJB) equation. The martingale method is developed by Cox and Huang (1989), which comes from the idea of global optimization by solving a dual problem.
For a robust optimization, the ideas of these two methods still work.
By the concept of G-Brownian motion (cf. Peng (2007, 2010)), the dynamic programming is still available for a robust problem, whose value function satisfies a Hamilton-Jacobi-Bellman-Isaacs (HJBI) equation. This method has been used in the fields of robust hedging problem, robust investment problem and robust investment-consumption problem, such as Tevzadze et al. (2013), Lin and Riedel (2014), Fouque et al. (2016) and Biagini and P{\i}nar (2017).
By introducing the conditional risk mapping as a dynamic analogues of coherent risk, the dynamic programming still can be established for this convex risk measure (cf. Artzner et al.(1999), Shapiro and U{\u{g}}urlu (2016), U{\u{g}}urlu (2017a)). Under a general uncertainty set with a reference probability measure, U{\u{g}}urlu (2018) regards a robust optimal investment problem as a maximization problem with respect to the conditional risk mapping and then solves this problem explicitly by the robust dynamic programming equation.
However, these two approaches derived from the dynamic programming are not adequate to solve the robust problem with jumps.
Based on the idea of global optimization, a robust problem can be solved by finding its saddle point. Neufeld and Nutz (2018) characterize an optimal strategy by a saddle point of a deterministic function and accomplish the martingale argument under both logarithmic and power utilities.
U{\u{g}}urlu (2017b) directly analyzes the objective function and then obtains a saddle point by Sion's minimax theorem.

\vskip 5pt
This paper is devoted to solving the robust investment-consumption problem by the martingale method, under both CRRA and CARA utilities. We would like to name the method of Neufeld and Nutz (2018) as the martingale method of robust optimization because it is the martingale property that establishes the equivalence between a deterministic optimization and a stochastic optimization.
Following this deterministic-to-stochastic paradigm, we introduce a deterministic functional (global kernel, cf. Definitions \ref{def G_power} and \ref{def G_CARA} below), then an optimal policy and a worst-case measure can be generated from its saddle point.
There are three contributions in this paper: solving the robust investment-consumption problem with jump diffusions; treating exponential utilities; considering time-varying confidence sets. Because of these novelties, the deterministic optimization problem, which is finding a saddle point of global kernel over a product space of functionals, becomes an infinite-dimensional problem. The difficulty is inherent and cannot be avoided by any techniques because it is derived from the appearance of intertemporal consumption and time-varying confidence sets.

\vskip 5pt
The rest of this paper is organized as follows. Section 2 presents the general model of the consumption-investment problem as well as the uncertainty set of semimartingale measures. Section 3 and Section 4 solve the robust consumption-investment problem for CRRA and CARA utilities, respectively, by martingale method. In Section 5, saddle points of global kernels for CRRA and CARA utilities are found without using Sion's minimax theorem. Section 6 concludes. All proofs are in  Appendixes \ref{app1}, \ref{app monotony and conclusion}, \ref{app measurable selector} and \ref{app opt c}.

\vskip 15pt
\setcounter{equation}{0}
\section{Robust Optimization}

\subsection{General Formulation}
A general robust problem can be formulated as
\begin{equation}\label{general robust}
\begin{aligned}
\sup_{\xi \in \mathfrak{A}} \inf_{P\in\mathfrak{P}}J(\xi, P; w_0),
\end{aligned}
\end{equation}
where $J$ is the objective function, $\mathfrak{A}$ is the admissible policy set and $\mathfrak{P}$ is the uncertainty set. Below, we will introduce the definitions of $J$, $\mathfrak{A}$, and $\mathfrak{P}$ successively.

The objective function $J$ is the expectation of the sum of cumulative utilities and terminal utility.
\begin{equation}\label{obj fun}
\begin{aligned}
J(\xi, P; w_0) \triangleq\E^P\left[ \int_0^T U(C^{\xi}_t) dt  + U(W^{\xi}_T)  \right],
\end{aligned}
\end{equation}
where the expectation is taken under the measure $P$. $w_0$ is the initial wealth, which is fixed and can be omitted for convenience. $C^{\xi}_t$ is the amount of consumption at time $t$, and $W^{\xi}_T$ is the terminal wealth. They are all determined by the policy $\xi$ as well as $w_0$. Here, we do not explain the rule that how $\xi$ determines consumption process and wealth process in detail, because the rule is different between the cases of CRRA and CARA utilities, respectively, in Sections 3 and 4. The utility function $U: \mathcal{D}\rightarrow \R$ reflects the investor's preference, which is usually concave and strictly increasing due to the risk aversion. For a given $w_0$, we denote the value of \eqref{general robust} as
\begin{equation}\label{value fun}
\begin{aligned}
u(w_0)= \sup_{\xi \in \mathfrak{A}} \inf_{P\in\mathfrak{P}}J(\xi, P; w_0),
\end{aligned}
\end{equation}
thus $u$ is the value function. We call $\xi \in \mathfrak{A}$ optimal if it attains the supremum in \eqref{value fun} and $P\in\mathfrak{P}$ worst-case if attaining the infimum.

\begin{Def}\label{def strategy}
	A policy $\xi$ is admissible if and only if
	\begin{enumerate}[(i)]
		\item $\xi$ is predictable,
		\item $W^{\xi}_{T}, C^{\xi}_{t}\in \mathcal{D}$, $\forall t\geq0$, $\mathfrak{P}$-q.s,
	\end{enumerate}
	where $\mathfrak{P}$-q.s. is an abbreviation for $\mathfrak{P}$ quasi surely, which means a property holds $P$-a.s. for all $P\in \mathfrak{P}$. The admissible policy set $\mathfrak{A}$ is the set of all admissible policies.
\end{Def}
There are two essential conditions for admissible policy. The first one means the policy is established on the information from the past. The second one requires the consumption process and the terminal wealth both take values in the domain of the utility. Otherwise, the objective function would be negative infinity. For investors with CRRA utilities, the second condition prohibits the bankruptcy of terminal wealth and an injection of cash as a negative consumption.

\subsection{\bf Model uncertainty}
We choose a Skorokhod space $\Omega= D_0([0,T],\mathbb{R}^d)$ as the collection of all c\`{a}dl\`{a}g path $\{\omega_t\}_{t\in[0,T]}$ starting at origin and then there exists a natural $\sigma$-field $\F$ by the Skorokhod topology. The canonical process $X_t(\omega)=\omega_t$ can generate the natural filtration on $(\Omega, \F)$. Same as Foldes (1990), we use $X$ to describe the log-price processes of stocks in the market and the model of the market is formulated by the measure on $\Omega$.
Denote $\mathfrak{P}(\Omega)$ as the Polish space of all measures on $\Omega$ and
$\mathfrak{P}$ as the uncertainty set of all possible measures, then $\mathfrak{P}$ is much smaller than $\mathfrak{P}(\Omega)$ due to a priori information about drift, volatility and jump. Referring to the model of Chen and Epstein (2010) and Epstein and Ji (2014), we firstly introduce a correspondence $\Theta$ for defining $\mathfrak{P}$.

Let $\R^d$ be the set of all $d$-dimensional vectors with Euclidean metric $d_2$ and $\mathbb{S}_+^d$ be the set of all $d\times d$ symmetric positive definite matrices with metric $d_2$ induced from 2-norm. For any $\varepsilon\in[0,2]$, denote
\begin{equation}\nonumber
\begin{aligned}
\mathcal{L}^{\varepsilon}\triangleq\left\{\mu:  \mu(\{0\})=0, \int_{\R^d} (|z|^{2-\varepsilon} \wedge 1) \mu(dz) < +\infty\right\}
\end{aligned}
\end{equation}
as a subset of all L\'{e}vy measures and let $\mathcal{L} \triangleq \mathcal{L}^{0}$ be the set of all L\'{e}vy measures. For any $\varepsilon\in[0,2]$, we can define a metric on $\mathcal{L}$ by
\[
d_{\mathcal{L}}^{\varepsilon}(\mu,\nu) = d_{BH}^{~\varepsilon}(|x|^{2-\varepsilon}\wedge1.\mu, |x|^{2-\varepsilon}\wedge1.\nu),~\forall \mu,\nu\in\mathcal{L},
\]
where $|x|^{2-\varepsilon}\wedge1.\mu$ is the measure defined by
\[
A  \mapsto \int_A |x|^{2-\varepsilon}\wedge1 \mu(dx), ~\forall A\in\mathcal{B}(\R^d),
\]
and $d_{BH}^{~\varepsilon}$ is the metric induced by $(\varepsilon\wedge1)$-H\"{o}lder continuous functions, i.e., for any measure $\mu$ and $\nu$,
\[
d_{BH}^{~\varepsilon}(\mu,\!\nu) \!\triangleq\! \sup\left\{\!\int_{\R^d} \!f d(\mu\!   -\!\nu): f\!\in\! C_b(\R^d),~\sup_{z\neq\hat{z}} \left[ |f(z)|\!\vee\!  \frac{|f(z) \!-\! f(\hat{z})|}{|z- \hat{z}|^{\varepsilon\wedge1}}\right]\leq1 \right\}.
\]
The subscript BH and superscript $\varepsilon$ stand for boundedness and H\"{o}lder continuity with exponent $\varepsilon\wedge1$. When $\varepsilon=0$, the metric space $(\mathcal{L}, d_{\mathcal{L}}^{0})$ is exactly the space studied by Neufeld and Nutz (2014). For any positive $\varepsilon$, $(\mathcal{L}, d_{\mathcal{L}}^{\varepsilon})$ is a generalized metric space whose metric can be infinity on the set $\mathcal{L}\setminus \mathcal{L}^{\varepsilon}$, while $(\mathcal{L}^{\varepsilon}, d_{\mathcal{L}}^{\varepsilon})$ is a traditional metric space. Specifically, when $\varepsilon>0$, $d_{BH}^{~\varepsilon}$ is the Kantorovich-Rubinshtein metric with $\varrho(z, \hat{z}) = |z- \hat{z}|^{\varepsilon\wedge1}$ (cf. theorem 8.3.2, Bogachev (2007)) and thus the customized metric $d_{\mathcal{L}}^{\varepsilon}$ induces the weak convergence on $\mathcal{L}^{\varepsilon}$.

Therefore, the product space $\R^{d}\times \mathbb{S}_+^d \times \mathcal{L}$ contains all L\'{e}vy triplets. Let $\mathcal{C}$ be a compact subset of $\R^{d}\times \mathbb{S}_+^d \times \mathcal{L}$ under the maximum metric:
\begin{equation}\nonumber
\begin{aligned}
&d_\mathcal{C}^{\varepsilon}( (y,M,\mu), (\hat{y},\hat{M},\hat{\mu}) )
\triangleq  d_2( y, \hat{y})\vee d_2( M, \hat{M})\vee d_\mathcal{L}^{\varepsilon}( \mu, \hat{\mu}) ,\\
&\forall (y,M,\mu), (\hat{y},\hat{M},\hat{\mu})\in \R^{d}\times \mathbb{S}_+^d \times \mathcal{L},
\end{aligned}
\end{equation}
for some $\varepsilon\in(0,2]$. Thus $(\mathcal{C}, d_\mathcal{C}^{\varepsilon})$ is a separable metrizable space and has finite elements in $\mathcal{L}\setminus \mathcal{L}^{\varepsilon}$.
Generally speaking, the condition $\varepsilon\in (0,2]$ guarantees the existence of a measurable saddle point of global kernel, which will be thoroughly explained in Section 5 again.

Let $\Theta$ be a weakly measurable correspondence from $[0,T]$ to $\mathcal{C}$, describing the range of all possible differential characteristics. The definition of weakly measurable correspondence is given by Aliprantis and Border (2006), definition 18.1, which is recall here for convenience.
\begin{Def}[Aliprantis and Border (2006), 18.1 Definition]\label{def weakly measurable correspondence}
Let $(S, \Sigma)$ be a measurable space and $X$ a topological space. We say that a  correspondence $\varphi:S \twoheadrightarrow X$ is weakly measurable, if $\varphi^l(G)\in\Sigma$  for each open subset $G$ of $X$,
where $\varphi^l$ is the lower inverse (also called the weak inverse) of the correspondences $\varphi$ and is defined by
\[
\varphi^l(A) = \{ x\in X :  \varphi(x) \cap A \neq \varnothing\}.
\]
\end{Def}
Any measurable selector from correspondence $\Theta$ is called the possible PII triplet because it has three components and
corresponds to a possible PII measure introduced below.
For convenience in notation, we give a convention that $\Theta$ also represents the set of all possible PII triplets, i.e.,
\[
\{ \theta: [0,T] \rightarrow \mathcal{C}| ~\mathrm{ Borel~measurable,}~ \theta_t\in \Theta_t, \forall t\in[0,T]\}.
\]
This convention about $\Theta$ would not cause any misunderstanding because one is a set and the other is a set-valued function. In previous studies of robust optimization, the confidence set $\Theta_t$ is usually independent of the time $t$, therefore, $\Theta$ is a constant correspondence. The problem with constant correspondence is trivial because it can be easily solved by proposing the local kernel without considering the measurability of the optimal policy and the worst-case PII triplet.

\begin{Def}\label{def uncertainty set}
$P$ is a possible measure if and only if $P\in \mathfrak{P}(\Omega)$ and under $P$,
\begin{enumerate}[(i)]
\item $X$ is a semimartingale with predicable characteristics $(\beta^P,  \alpha^P, \nu^P)$;
\item there exists a $\theta^P:  [0,T] \times \Omega \rightarrow \mathcal{C}$  such that $\theta^P = (b^P, \Sigma^P , F^P ) \in \Theta$ and
    \begin{equation}\nonumber
    \begin{aligned}
    \beta^P_t = \int_0^t b^P_s  ds, ~
    \alpha^P_t = \int_0^t \Sigma_s^P  ds, ~
    \nu^P(ds,dy)= F^P_s(dy)ds,
    \end{aligned}
    \end{equation}
    almost surely.
\end{enumerate}
The triplet $\theta^P =(b^P, \Sigma^P, F^P)$ is called the DC (differential characteristics) triplet of $X$ under $P$ and $\mathfrak{P}$ is the set of all possible measures.
\end{Def}
Definition \ref{def uncertainty set} describes how to parameterize the model uncertainty by a correspondence $\Theta$. The uncertainty set contains all the semimartingale measures, under which the DC triplet of $X$ is a PII triplet almost surely. For a possible measure $P$ whose associated $\theta^P$ is independent of the probability space, we call it a possible PII measure because the log-price processes are processes with independent increments under this measure (cf. theorem \uppercase\expandafter{\romannumeral2}.4.15, Jacod and Shiryaev (2013)).
The set of all possible PII measures is
\begin{equation}\label{def of P0}
\begin{aligned}
\mathfrak{P}_0  \triangleq \{ P\in\mathfrak{P}: \exists~\theta\in \Theta,  \theta^P(\omega) = \theta, ~\forall \omega\in\Omega\},
\end{aligned}
\end{equation}
which has a one-to-one relationship with $\Theta$, the set of all possible PII triplets.
Finally, to make the uncertainty set have better properties, we give assumptions on $\Theta$ as the end of this section.

\begin{Assumption}\label{ass_2.2}
For each $t$, denote
\begin{equation}\nonumber
\begin{aligned}
\mathbf{S}_t  \triangleq \bigcup_{(y,M,\mu)\in \Theta_t} \mathrm{supp}(\mu)
\end{aligned}
\end{equation}
as the union of all supports of L\'{e}vy measures in $\Theta_t$. Suppose $\mathbf{S}_t $ is closed, bounded, and non-degenerate, i.e., there is a positive number $\kappa_t$ such that
\begin{equation}\nonumber
\begin{aligned}
\{z: |z|\leq \kappa_t^{-1}\}\subseteq \mathrm{Conv}(\mathbf{S}_t\cup{0})\subseteq\{z: |z|\leq \kappa_t\},
\end{aligned}
\end{equation}
where $\mathrm{Conv}(\cdot)$ represents the convex hull of a set.
\end{Assumption}

Assumption \ref{ass_2.2} formulates the set $\mathbf{S}_t$ of all possible jump sizes and demands $\mathbf{S}_t$  is a closed set, which is also bounded and non-degenerate by a number $\kappa_t$. Though the support of any L\'{e}vy measure is closed, it is not trivial to assume $\mathbf{S}_t$ is closed because it can be an infinite union.
The closeness of $\mathbf{S}_t$ actually contributes to the existence of an optimal policy over a nonclosed admissible policy set.
In the definition of non-degeneracy, we require the convex hull of $\mathbf{S}_t\cup{0}$ contains a neighborhood of origin, because the set $\mathbf{S}_t$ may not be convex in general. Moreover, in the case of CRRA utility (see \eqref{bound of O_t} below), it is enough to deduce the boundedness of admissible investment policy through the constraint on $\mathrm{Conv}(\mathbf{S}_t\cup{0})$.
The boundedness of jumps simplifies the differential notation for $X$  by omitting the truncation function (cf. definition \uppercase\expandafter{\romannumeral1}.2.6, Jacod and Shiryaev (2013)). For each $P\in\mathfrak{P}$,
\begin{equation}\label{sde of X}
\begin{aligned}
dX_t&=  b^P_t dt + \sigma^P_t dB_t +\int_{\R^d}z \tilde{\nu}^P(dz,ds),
\end{aligned}
\end{equation}
where $(b^P, \Sigma^P, F^P)$ is the DC triplet of $X$ under $P$ and $\tilde{\nu}^P$ is the compensated random measure with compensator $\nu^P$. $\sigma^P_t$ is a $d\times d$ lower triangular matrix with positive diagonal entries and satisfies $\sigma^P_t\times (\sigma^P_t)^T = \Sigma^P_t$, which is unique by Cholesky decomposition (cf. Golub and Van Loan (2012)). $\{B_t\}_{t\geq0}$ is the corresponding $P$-Brownian motion with respect to $\sigma^P$.

\begin{Assumption}\label{ass_2.1}
$\Theta$ is closed-valued and convex-valued, i.e., $\Theta_t$ is a closed and convex set for every $t\in[0,T]$.
\end{Assumption}
Assumption \ref{ass_2.1} requires each confidence set is convex and closed, which is common in an optimization problem. If $\Theta_t$ is not a singleton, $\Theta_t$ must be contained in $\R^{d}\times \mathbb{S}_+^d \times \mathcal{L}^\varepsilon$, otherwise $\mathcal{C}$ has infinitely many elements in $\mathcal{L}\setminus \mathcal{L}^{\varepsilon}$ and is not compact. Furthermore, since $\mathcal{C}$ is compact, $\Theta_t$ is bounded convex and compact, which contributes to the existence of a worst-case PII triplet.
In the range $(0,2]$ of $\varepsilon$, there are two significant values. When $\varepsilon=1$, the log-price processes have jumps of finite variation under each possible measure by theorem 21.9 of Sato (1999). When $\varepsilon=2$, under each possible measure, the log-price processes are compound non-homogeneous Poisson processes with finite activity by theorem 21.3 of Sato (1999).
Finally, for avoiding redundant notations, we further assume the $\kappa_t$ in Assumption \ref{ass_2.2} is the bound of each $\Theta_t$, i.e.,
\[
  \sup_{(y,M,\mu) \in \Theta_t}\{ d_2(y , 0) \vee d_2(M , 0) \vee d_\mathcal{L}^\varepsilon(\mu , 0)\}\leq \kappa_t.
\]

\vskip 15pt
\setcounter{equation}{0}
\section{CRRA Utilities}
We consider a family of CRRA utilities
\begin{equation}\label{def of CRRA}
U(x)= \left\{
\begin{aligned}
&\log(x),&   p=0,\\
&\frac{x^p}{p}, & p\in (-\infty,0)\cup(0,1),
\end{aligned}
\right.
\end{equation}
with coefficient $p$ and domain $\mathcal{D} = (0,+\infty)$. $1-p$ has significance as the coefficient of relative risk aversion (cf. Pratt (1964)).
This section consists of three parts. The first one specifies the rule of policy, that is, how to determine the consumption process and terminal wealth by an admissible policy. Then we introduce a deterministic function called global kernel and show its relationship with the objective function. Finally, a saddle point of global kernel generates an optimal policy and a worst-case measure of the robust consumption-investment problem.

\subsection{Admissible Policy}
Under CRRA utility, it is usual to choose the percentage of the wealth invested in stocks and the amount of consumption per time unit to maximize the objective function, such as Merton (1969) and Foldes (1990). However, based on the remark 2.1 of Nutz (2010), we set the policy $\xi=(\pi, c)=(\{\pi_t\}_{t\in[0,T]}, \{c_t\}_{t\in[0,T]})$, where $\pi_t$ is the investment amount to wealth ratio and $c_t$ is the consumption amount to wealth ratio at time $t$. Under this rule, the wealth process $\{W_t^{(\pi,c)}\}_{t\in[0,T]}$ is the solution of
\begin{equation}\label{sde of W_pic}
\left\{
\begin{aligned}
dW^{(\pi,c)}_t&=   \pi_t W^{(\pi,c)}_t dX_t - c_tW^{(\pi,c)}_tdt,\\
W^{(\pi,c)}_0 &= w_0,
\end{aligned}
\right.
\end{equation}
and the amount of consumption at time $t$ is  $C_t^{(\pi, c)} =  c_t W_t^{(\pi, c)}$.
From the second condition in Definition \ref{def strategy},  the consumption ratio $c_t$ must be non-negative and the investment ratio must satisfy
\[
\pi_t ^Tz>-1, \forall z\in \mathbf{S}_t, t\in[0,T].
\]
The latter is a direct conclusion from the theorem I.6.61 of Jacod and Shiryaev (2013) and the discussion of Karatzas and Kardaras (2007). Hence the admissible policy set $\mathfrak{A}$ can be expressed elegantly by restricting the policies.
\begin{equation}\nonumber
\!\!\!\!\begin{aligned}
\mathfrak{A}\! =\! \big\{&(\bar{\pi}, \bar{c}) \!:\! [0,\!T]\!\times  \Omega \! \rightarrow\! \R^d\!\!\times\!\R_+ \big| \text{ predictable},\!\!
&\bar{\pi}_t \in \mathcal{O}_t, \forall t\in[0,\!T], \mathfrak{P}\text{-q.s.} \big\},
\end{aligned}
\end{equation}
where $\mathcal{O}_t \triangleq \{y\in\R^d : x^Tz>-1, \forall z\in \mathbf{S}_t  \}$.
Recalling the condition of non-degeneracy in Assumption \ref{ass_2.2}, we have
\begin{equation}\label{bound of O_t}
\begin{aligned}
\mathcal{O}_t
=& \left\{x\in\R^d \!: x^Tz\!>\!-1, \forall z\in \mathbf{S}_t  \right\}\\
=& \left\{x\in\R^d \!: x^Tz\!>\!-1, \forall z\in \mathrm{Conv}(\mathbf{S}_t\cup{0})  \right\}\\
\subseteq& \left\{x\in\R^d \!: x^Tz\!>\!-1, \forall z\in \{x\in \R^d: d_2(x,0)\!\leq\! \kappa_t^{-1}\}  \right\}\\
=& \left\{x\in \R^d\!: d_2(x,0)\!<\! \kappa_t\right\},
\end{aligned}
\end{equation}
thus $\mathcal{O}_t $ is bounded by $\kappa_t$. Moreover, it can be verified that $\mathcal{O}_t $ is a bounded convex set and contains the origin.

\subsection{Global kernel for CRRA utilities}
The global kernel $G$ is a functional defined on $\mathfrak{A}_0\times \Theta$ and is closely related to the objective function (see Lemma \ref{J=G} below). Firstly, we introduce the deterministic policy set $\mathfrak{A}_0$ as a domain of the global kernel.

\begin{Def}\label{def A0}
\begin{equation}\nonumber
\begin{aligned}
\mathfrak{A}_0 \triangleq \big\{&(\pi,c): [0,T] \rightarrow \R^d\times \R_+ ~\big|
~\mathrm{measurable}, \pi_t\in \mathcal{O}_t, \forall t\in[0,T] \big\}
\end{aligned}
\end{equation}
is the set of all deterministic admissible policies.
\end{Def}

\begin{Def}\label{def genertae strategy}
$(\tilde{\pi}, \tilde{c})\in\mathfrak{A}$ is called the generated policy from $(\pi,c)$, if $(\pi,c)\in\mathfrak{A}_0$ and
\[
\tilde{\pi}_t(\omega)=\pi_t, \tilde{c}_t(\omega) = c_t, \forall t\in[0,T], \omega\in\Omega.
\]
We say a deterministic admissible policy is optimal, if its generated policy attains the supremum in \eqref{value fun}.
\end{Def}
$\mathfrak{A}_0$ and $\mathfrak{A}$ have a close connection while $\mathfrak{A}_0$ is independent of the probability space. Retrospecting the expression of $\mathfrak{A}$, we notice $(\bar{\pi}(\omega), \bar{c}(\omega))\in \mathfrak{A}_0$ quasi surely for any $(\bar{\pi}, \bar{c}) \in \mathfrak{A}$. Meanwhile, any deterministic admissible policy in $\mathfrak{A}_0$ can generate an admissible policy in $\mathfrak{A}$ by Definition \ref{def genertae strategy}. Due to their close relationship, we can regard the elements in $\mathfrak{A}_0$ as in $\mathfrak{A}$ if there is no misunderstanding.

\begin{Def}\label{def g CRRA}
The local kernel (at time $t$) is a function defined on
$ \mathcal{O}_t\times \Theta_t$:
\begin{equation}\nonumber
\begin{aligned}
&g_t^{(y, M, \mu)}(x)
\!\triangleq\!
x^Ty \!-\! \frac{1\!-\!p}{2}x^T M x +\!\! \int_{\R^d}\!\!\! \left(U(1\!+\! x^T z)\!-\! U(1) \!-\! x^T z\right) \!\mu(dz),
\end{aligned}
\end{equation}
where $x\in\mathcal{O}_t$ and $(y, M, \mu)\in \Theta_t$.
\end{Def}

\begin{Def}\label{def G_power}
The global kernel $G$ (over period $[0, T]$) is a two-variable function defined on $\mathfrak{A}_0\times \Theta$. For every $(\pi,c)\in \mathfrak{A}_0$ and $\theta\in\Theta$,
\begin{equation}\nonumber
\begin{aligned}
G((\pi,c),\theta)
\!\triangleq\!
\left\{
\begin{aligned}
&\!\!\int_0^T \!\!\!\bigg(\! \int_0^t\! (g^{\theta_s}(\pi_s) \!-\!c_s) ds \!+\!\!g^{\theta_t}(\pi_t) \!-\!c_t+U(c_t)\bigg)dt, &p=0,\\
&\!\!\int_0^T\!\!\! \exp\!\bigg( \!\int_0^t\! \!\!p(g^{\theta_s}(\pi_s) \!-\!c_s) ds\!\bigg) \bigg(\!g^{\theta_t}\!(\pi_t) \!-\!c_t\!+\!U(c_t)\! \bigg) dt, \!\!&\!\!p\neq0.
\end{aligned}
\right.
\end{aligned}
\end{equation}
\end{Def}
Observing the expression of $g_t$ in Definition \ref{def g CRRA}, we notice that the domain $\mathfrak{A}_0\times \Theta$ depends on time $t$, but the function $g_t$ is independent. So we omit the subscript of $g_t$ as in Definition \ref{def G_power}.
The global kernel can be expressed by the policy and PII triplet directly, not through any intermediate variables, such as the consumption process and the terminal wealth. This is the advantage of the global kernel in contrast to the objective function.
The motivation of proposing global kernel comes from  the martingale equality in the following lemma.
\begin{Lemma}\label{J=G}
For any $(\bar{\pi}, \bar{c}) \in \mathfrak{A}$ and any $P\in\mathfrak{P}$, the objective function \eqref{obj fun} can be expressed by $(\bar{\pi}, \bar{c})$, $\theta^P$ and $P$ directly.
\begin{equation}\nonumber
\begin{aligned}
&J((\bar{\pi}, \bar{c}),P)\!
=\!\!\left\{\!\!
\begin{aligned}
&(T\!+\!1)\log(w_0) \!+\!\int_\Omega  G((\bar{\pi},\bar{c}),\theta^P) P(d\omega),&p=0,\\
&w_0^p\bigg(\frac{1}{p}\!+\!
\int_\Omega  G((\bar{\pi},\bar{c}),\theta^P) Q^{(\bar{\pi},\theta^P)}(d\omega)\bigg),&p\neq0.
\end{aligned}
\right.
\end{aligned}
\end{equation}
For $p\neq0$, $Q^{(\bar{\pi},\theta^P)}$ is the equivalent measure whose Radon-Nikodym derivative is
\begin{equation}\nonumber
\begin{aligned}
\frac{dQ^{(\bar{\pi},\theta^P)}}{dP}(T)&= \exp\left\{ \int_0^T p\bar{\pi}_t^T \sigma^P_tdB_t- \frac{1}{2}p^2\bar{\pi}_t^T\Sigma^P_t\bar{\pi}_t dt \right\}\\
&\times\exp\!\bigg\{\!\! \int_0^T\!\!\int_{\R^d}\!p\log(1+\bar{\pi}_t^Tz)\tilde{\nu}^P(dz,dt) \\ &+\!\int_0^T\!\!\int_{\R^d}\!\big( p\log(1\!+\!\bar{\pi}_t^T z)+1\!-\!(1\!+\!\bar{\pi}_t^Tz)^p\big) F^P_t(dz)dt\!\bigg\}.
\end{aligned}
\end{equation}
\end{Lemma}

\begin{proof}
See Appendix \ref{app1}.
\end{proof}

We primarily use the property of exponential martingales to rewrite the objective function in Lemma \ref{J=G}.
In Equation \eqref{obj fun}, the objective function $J$ depends on the policy $\xi=(\bar{\pi}, \bar{c})$ through the intermediate variables $\{C^{\xi}_t\}_{t\in[0,T]}$ and $W^{\xi}_T$. While in Lemma \ref{J=G}, it is  expressed as a function of $(\bar{\pi}, \bar{c})$, $\theta^P$, and $P$ without any intermediate variables by using global kernel. Above equality of objective function and global kernel is very important in establishing the relationship between saddle points of $J$ and $G$.

\subsection{Optimal policy and worst-case measure}
We present an assumption for the confidence sets and then demonstrate the existence of an optimal policy and a worst-case measure of the robust consumption-investment problem.  For concentrating on the robust problem, we postpone the study about global kernel in Section 5 and acquiesce to the existence of its saddle point.
\begin{Assumption}\label{ass_CRRA}
For any $p<0$ and $t\in[0,T]$,
\begin{equation}\nonumber
\begin{aligned}
(b_t, \Sigma_t, F_t)\in\Theta_t \Rightarrow b_t^T\Sigma_t^{-1}b_t\leq \frac{2(1-p)^2}{-p}.
\end{aligned}
\end{equation}
\end{Assumption}

    \begin{figure}[htbp]
    \centering
    \includegraphics[width=0.8\textwidth]{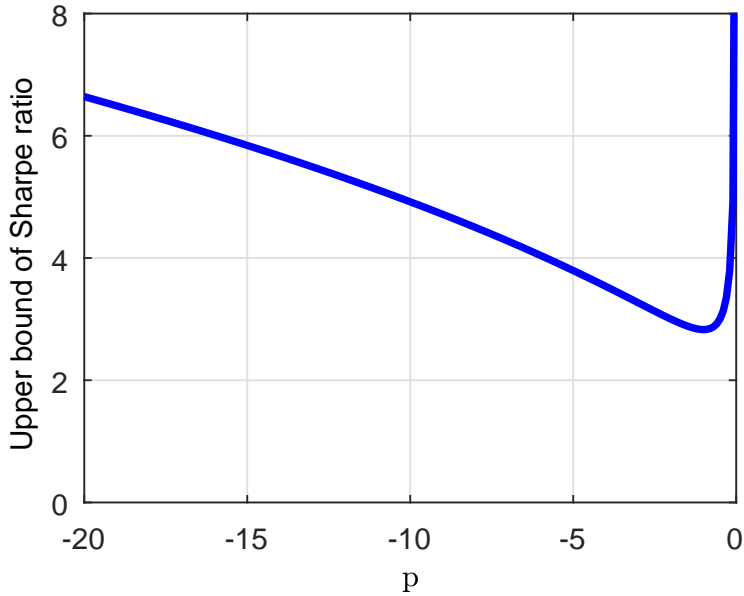}
    \caption{The upper bound of Sharpe ratio for each $p<0$.}
    \label{fig crra}
    \end{figure}

The condition in Assumption \ref{ass_CRRA} is relaxed enough in the real market. Since $((b_t^P)^T(\Sigma^P_t)^{-1}b_t^P)^{\frac{1}{2}}$ is the Sharpe ratio of the diffusion process under measure $P$, Assumption \ref{ass_CRRA} indeed gives the upper bounds of Sharpe ratio under all possible measure for every $p$. The upper bound is drawn in Figure \ref{fig crra} and attains the minimum as $2\sqrt{2}$ at $p=-1$. For a market with Sharpe ratio $2\sqrt{2}$, the probability of a negative yield is $\Phi(-2\sqrt{2})\approx 0.23\%$ after a year of investment.
According to the annual returns of S\&P 500 (\^{}GSPC) Index from Jan 01, 1950 to Jan 01, 2018,
there are 17 years that have negative returns, thus the probability of a negative yield is $17/68 = 25\%$, which is much greater than  $0.23\%$. Referring to the Table C1 of Frazzini et al. (2013), we notice that the Sharpe ratio of overall US stocks is 0.39, and the Sharpe ratio of Buffett performance is about 0.7, which are both less than $2\sqrt{2}$. Hence, Assumption \ref{ass_CRRA} is  undemanding and realistic. It plays an important role in finding a saddle point of the global kernel (see Section 5 below). Conceptually, it precludes the good market in which over-consumption occurs and thus simplifies the matters considered in kernel analyses.

As the end of this section, we present the main result of CRRA utilities that a saddle point of global kernel can generate an optimal policy and a worst-case measure.

\begin{theorem}\label{thm_CRRA}
There exist an optimal policy and a worst-case measure of the robust consumption-investment problem, denoted by $(\tilde{\pi}^*, \tilde{c}^*)$ and $P^*$. Moreover,
\begin{enumerate}[(i)]
\item $((\tilde{\pi}^*, \tilde{c}^*), P^*)$ is a saddle point of objective function $J$;
\item $(\tilde{\pi}^*, \tilde{c}^*)$ is the generated policy from a $({\pi}^*, {c}^*)\in \mathfrak{A}_0$;
\item $P^*$ is the possible PII measure whose DC triplet equals $\theta^*$;
\item $(({\pi}^*, {c}^*), \theta^*)$ is a saddle point of global kernel $G$;
\item the value function is
\begin{equation}\nonumber
\begin{aligned}
u(w_0)=\!\!\left\{
\begin{aligned}
&(T\!+\!1)\log(w_0) \!+\! G((\pi^*,c^*),\theta^*),&p=0,\\
&\frac{w_0^p}{p}\!+\!w_0^p
  G((\pi^*,c^*),\theta^*),&p\neq0.
\end{aligned}
\right.
\end{aligned}
\end{equation}
\end{enumerate}
\end{theorem}

\begin{proof}
The existence of a saddle point of global kernel is supported by Theorem \ref{exchange} below.
In Appendix \ref{app1}, we prove the $(\tilde{\pi}^*, \tilde{c}^*)$ and $P^*$, which satisfy conditions $(ii)$ and $(iii)$, constitute a saddle point of the objective function.
\end{proof}

Theorem \ref{thm_CRRA} shows that a saddle point of the global kernel can generate a saddle point of objective function, that is, a solution of the robust consumption-investment problem. Also, the value function is expressed by the value of $G$ at saddle point. In general, a robust stochastic control can be solved by an auxiliary deterministic minimax problem, which follows the deterministic-to-stochastic paradigm and can be named as the martingale method of robust optimization.

\vskip 15pt
\setcounter{equation}{0}
\section{CARA Utilities}
In this section, we consider a family of CARA utilities:
\[
U(x)= \frac{1}{-a}e^{-a x},
\]
whose domain is $(-\infty,+\infty)$ and parameter $a$ is positive.
Because the coefficient of absolute risk aversion of $U$ (cf. Pratt (1964)) is
constant, people usually regard the policy $\xi$ as a two-tuples $(\Pi, C)$, where $\Pi_t$ is the amount of investment in stocks and $C_t$ is the amount of consumption at time $t$, e.g., Karatzas et al. (1987), Vila and Zariphopoulou (1994), Liu (2004) and Chen et al. (2012).

Noticing CARA and CRRA utilities are both special cases of Hyperbolic absolute risk aversion (HARA) utilities, we want to use the martingale method as for CRRA utilities to solve the case of CARA utilities. For this purpose, we define a new kind of policy $\xi=(\Pi, D)=(\{\Pi_t\}_{t\in[0,T]}, \{D_t\}_{t\in[0,T]})$. Under this policy, the consumption amount per time unit is
$$
C_t^{(\Pi, D)} = D_t+ q_t W^{(\Pi, D)}_t,~\forall t\in[0,T],
$$
and the dynamic of wealth process $\{W_t^{(\Pi,D)}\}_{t\in[0,T]}$ is
\begin{equation}\label{sde of W_PiD}
\left\{
\begin{aligned}
dW^{(\Pi,D)}_t&=   \Pi_t dX_t - (D_t+ q_t W^{(\Pi, D)}_t) dt,\\
W^{(\Pi,D)}_0 &= w_0,
\end{aligned}
\right.
\end{equation}
where $~q_t = (T-t+1)^{-1}$ is a scale coefficient about time $t$.
According to Theorem \ref{thm_CRRA}, the quantity $q_t$ is the optimal ratio of consumption at time $t$ for an investor with logarithmic utility so that $D= C^{(\Pi, D)}-qW^{(\Pi, D)}$ is the excess amount of consumption compared to an investor with logarithmic utility (hereinafter referred to as the excess consumption).
Proposing excess consumption contributes to a direct expression of objective function by using the policy and measure, just like Lemma \ref{J=G}. An analogous martingale equality will be established in Lemma \ref{J=G_CARA} below.

Before this, we introduce the admissible policy set $\mathfrak{A}$, the deterministic admissible policy set $\mathfrak{A}_0$, the local kernels $\{h_t\}_{t\in[0,T]}$, and the global kernel $H$ for CARA utilities successively. The notations are same as the ones in CRRA case if there is no essential difference between them.
The admissible policy set is
\begin{equation}\label{def of A_CARA}
\begin{aligned}
\mathfrak{A} \triangleq \big\{&(\bar{\Pi}, \bar{D}) : [0,T]\times \Omega \rightarrow \R^d\times\R ~\big|
\text{ predictable}\big\},
\end{aligned}
\end{equation}
which has less constraints than CRRA case because the domain of CARA utility is $\R$. Indeed, the $\mathcal{O}_t$ defined in Section 3 corresponds to the $\R^d$ in \eqref{def of A_CARA}.
Similarly as Definitions \ref{def A0} and \ref{def genertae strategy}, the deterministic admissible policy set is
\begin{equation}\nonumber
\begin{aligned}
\mathfrak{A}_0 \triangleq \big\{&(\Pi, D): [0,T] \rightarrow \R^d\times \R |\text{ measurable}\big\},
\end{aligned}
\end{equation}
and the way of generating an admissible policy from a deterministic admissible policy does not change.
\begin{Def}\label{def g_CARA}
The local kernel (at time $t$) is $h_t$ defined by
\[
\!\!h_t^{(y, M, \mu)}(x)= x y \!  -\! \frac{1}{2}a q_t x^T M x + \!\! \int_{\R^d} \!\!\left(\frac{e^{-a q_t x^T z}}{-aq_t} +\frac{1}{aq_t} - x^T z \right)\mu(dz),
\]
for any $x\in\R^d$ and $(y, M, \mu)\in \Theta_t$.
\end{Def}

\begin{Def}\label{def G_CARA}
The global kernel $H$ (over period $[0,T]$) is a two-variable function on $\mathfrak{A}_0\times \Theta$. For any $(\Pi,D)\in \mathfrak{A}_0$ and $\theta\in\Theta$,
\begin{equation}\nonumber
\begin{aligned}
H((\Pi,D),\theta)\!\!=\!\!\int_0^T\!\!\!\! \exp\bigg(\!\! \int_0^t\!\!\!\! -aq_s(h_s^{\theta_s}(\Pi_s)\!\! -\!D_s) ds\bigg)\! \bigg(q_t(h_t^{\theta_t}(\Pi_t) \!\!-\! D_t) \!+\! U(D_t)\!\bigg) dt.
\end{aligned}
\end{equation}
\end{Def}
We adopt new notations for the local kernels and global kernel here because they have great differences from the CRRA case.
Precisely, $h_t$ is a function of $q_t$, that is, a function of remaining duration $T-t$. It can been seen that $H$ and $G$ have similar expressions, where the $-a$, $U(D_t)$, and $q_t(h_t^{\theta_t}(\Pi_t) \!\!-\! D_t)$ correspond to the $p$, $U(c_t)$, and $(g^{\theta_t}(\pi_t) \!-\!c_t)$, respectively. Therefore,
the relationship between the global kernel and the objective function for CARA utilities can be obtained by imitating the procedure in CRRA case with negative $p$.

\begin{Lemma}\label{J=G_CARA}
For any $(\bar{\Pi}, \bar{D}) \in \mathfrak{A}$, $P\in\mathfrak{P}$,
\begin{equation}\nonumber
\begin{aligned}
J((\bar{\Pi}, \bar{D}),P)&= \exp\left( \frac{-aw_0}{T\!+\!1}\right)\bigg(\frac{-1}{a}+ \!\!\int_\Omega H((\bar{\Pi},\bar{D}),\theta^P)Q^{(\bar{\Pi},\theta^P)}(d\omega)\bigg),
\end{aligned}
\end{equation}
where $Q^{(\bar{\Pi},\theta^P)}$ is an equivalent measure whose Radon-Nikodym derivative
is
\begin{equation}\nonumber
\begin{aligned}
\frac{dQ^{(\bar{\Pi},\theta^P)}}{dP}(T)&= \exp\left\{-\!\!\int_0^T \!\!aq_t\bar{\Pi}^T_t\sigma^P_t dB_t -\!\int_0^T  \!\frac{1}{2}a^2q_t^2\bar{\Pi}^T_t\Sigma^P_t \bar{\Pi}_tdt \right\}\\
&\times\exp\!\bigg\{\!\int_0^T\!\!\!\int_{\R^d}\!-aq_t\bar{\Pi}^T_tz\tilde{\nu}^P(dz,dt) \\ &+\!\int_0^T\!\!\!\int_{\R^d}\!(-aq_t\bar{\Pi}^T_tz + 1\!-\!e^{-aq_t\bar{\Pi}^T_tz})F^P_t(dz)dt\!\bigg\}.
\end{aligned}
\end{equation}
\end{Lemma}
\begin{proof}
See Appendix \ref{app1}.
\end{proof}

\begin{Assumption}\label{ass_CARA}
For any $t\in[0,T]$,
\[
(b_t, \Sigma_t, F_t)\in \Theta_t ~\Rightarrow~ b_t^T\Sigma_t^{-1}b_t \leq  2q_t(1-\log(q_t)).
\]
\end{Assumption}
Just as Assumption \ref{ass_CRRA} in the CRRA case, we demand Assumption \ref{ass_CARA} to support the main result -- Theorem \ref{thm_CARA} for CARA utilities. Assumptions \ref{ass_CRRA} and \ref{ass_CARA} both give restrictions on the Sharpe ratio but still have difference.
The constraints in Assumption \ref{ass_CRRA} vary in parameter $p$ but the constraints in Assumption \ref{ass_CARA} are diverse in time instead of parameter $a$.
\begin{figure}[htbp]
	\centering
	\includegraphics[width=0.8\textwidth]{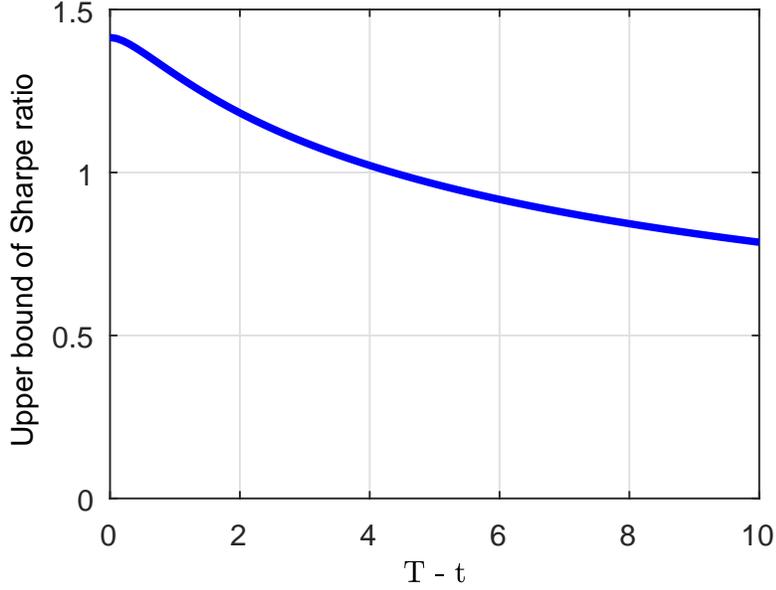}
	\caption{The upper bound of Sharpe ratio for different remaining duration $T-t$.}
	\label{fig cara}
\end{figure}
In Figure \ref{fig cara}, the upper bounds for every remaining durations  are depicted. It is obvious that the upper bound is decreasing with the increasing of $T-t$. When $T-t=5~(10)$, the upper bound of Sharpe ratio is about $0.96~(0.786)$ and thus the probability of a negative yield is about $16.9\%~(21.5\%)$ after a year of investment. This constraint is stricter than the condition in CRRA case but still tolerable enough for the vast majority of investment periods. Finally, we conclude this section by giving an optimal policy and a worst-case measure in the following theorem.

\begin{theorem}\label{thm_CARA}
There exist an optimal policy and a worst-case model of the robust consumption-investment problem, denoted by $(\tilde{\Pi}^*, \tilde{D}^*)$ and $P^*$. Moreover,
\begin{enumerate}[(i)]
\item $((\tilde{\Pi}^*, \tilde{D}^*), P^*)$ is a saddle point of objective function $J$;
\item $(\tilde{\Pi}^*, \tilde{D}^*)$ is the generated policy from $({\Pi}^*, {D}^*)\in\mathfrak{A}_0$;
\item $P^*$ is the possible PII measure whose DC triplet equals $\theta^*$;
\item $(({\Pi}^*, {D}^*), \theta^*)$ is a saddle point of global kernel $H$;
\item the value function is
\begin{equation}\nonumber
\begin{aligned}
u(w_0)= \exp\left( \frac{-aw_0}{T+1}\right)\left(-\frac{1}{a}+ H((\Pi^*,D^*),\theta^*)\right).
\end{aligned}
\end{equation}
\end{enumerate}
\end{theorem}
\begin{proof}
$(({\Pi}^*, {D}^*), \theta^*)$ exists by Theorem \ref{exchange CARA} below, and it is easy to verify $((\tilde{\Pi}^*, \tilde{D}^*), P^*)$ is a saddle point of $J$ by referring the proof of Theorem \ref{thm_CRRA}, while using Lemma \ref{J=G_CARA}.
\end{proof}

An optimal excess consumption $\tilde{D}^*$ is given in Theorem \ref{thm_CARA}, so an optimal amount of consumption for a CARA investor at time $t$ is
\[
(T-t+1)^{-1}\left(W^{(\tilde{\Pi}^*,\tilde{D}^*)}_t+ \tilde{D}^*_t\right).
\]
In kernel analyses, we will know $\tilde{D}^*$ is a non-negative function thus
a CARA investor will consume more than a logarithmic investor if they have the same amount of wealth. Meanwhile, more consumption amount may make the wealth negative, which is acceptable for a CARA investor but intolerable for a CRRA investor.

\vskip 15pt
\setcounter{equation}{0}
\section{Kernel Analyses}
In this section, we divide the policy into two components, an investment policy and a consumption policy, for studying the relationship between global kernel and local kernels. Therefore, the global kernel can be regarded as a function with three arguments: an investment policy, a consumption policy, and a PII triplet. Meanwhile, we substitute the expressions $G(\pi,c,\theta)$ and $H(\Pi, D, \theta)$ for $G((\pi,c),\theta)$ and $H((\Pi, D), \theta)$, respectively, to show $G$ and $H$ have three arguments.

The main contribution of this section is finding a saddle point of the global kernel, which can generate a saddle point of the objective function, i.e.,
an optimal policy and a worst-case measure of the robust consumption-investment problem (see Theorems \ref{thm_CRRA} and \ref{thm_CARA}). In Definitions \ref{def G_power} and \ref{def G_CARA},
the global kernels $G$ and $H$ have similar expressions, so the procedures of finding their saddle points are similar. The difficulty comes from the fact that the global kernel is a functional defined on infinite dimensional spaces, and this difficulty is inherent and cannot be avoided because it is caused by time-varying confidence sets and intertemporal consumption.
There are three steps in finding a saddle point. By proving the measurable saddle point theorem, we firstly find a deterministic admissible investment policy and a PII triplet, whose values at each time $t$ constitute a saddle point of the local kernel (cf. Theorems \ref{measurable selector} and \ref{measurable selector CARA}).
Then we obtain the candidate policy and PII triple by defining an appropriate consumption policy. Finally, we demonstrate that the chosen policy and PII triple constitute a saddle point of the global kernel (cf. Theorems \ref{exchange} and \ref{exchange CARA}).

\subsection{Kernel of CRRA}
We implement the steps described above for CRRA utilities under Assumptions \ref{ass_2.1}, \ref{ass_2.2}, and \ref{ass_CRRA}.
In order to regard the global kernel as a function with three variables, we rewrite its domain as
\[
\mathfrak{A}^\pi_0\times \mathfrak{A}^c_0 \times \Theta,
\]
where
\begin{equation}\nonumber
\begin{aligned}
\mathfrak{A}^\pi_0  \triangleq \big\{&\pi: [0,T] \rightarrow \R^d ~\big|
\text{ measurable}, \pi_t\in \mathcal{O}_t, \forall t\in[0,T] \big\},
\end{aligned}
\end{equation}
and
\begin{equation}\nonumber
\begin{aligned}
\mathfrak{A}^c_0  \triangleq \big\{&c : [0,T] \rightarrow \R_+
\big|
~\mathrm{measurable}
\big\}.
\end{aligned}
\end{equation}
Actually, $\mathfrak{A}_0$ in Definition \ref{def A0} is the Cartesian product of $\mathfrak{A}^\pi_0$ and $\mathfrak{A}^c_0$.

\begin{Def}\label{def monotonous}
	For any $(\pi, \theta),~(\hat{\pi}, \hat{\theta})\in \mathfrak{A}^\pi_0\times \Theta$, we say $(\pi, \theta) \succcurlyeq(\hat{\pi}, \hat{\theta})$ if and only if
	\[
	g_t^{\theta_t}(\pi_t) \geq g_t^{\hat{\theta}_t}(\hat{\pi}_t), \forall t\in[0,T].
	\]
	We call $(\pi, \theta)$ is greater than $(\hat{\pi}, \hat{\theta})$ in the kernel order.
\end{Def}

\begin{theorem}\label{g_monotony}
	For any $c \in \mathfrak{A}^c_0$ and $(\pi, \theta),(\hat{\pi}, \hat{\theta}) \in \mathfrak{A}^\pi_0\times\Theta$,
\[
(\pi, \theta)\succcurlyeq(\hat{\pi}, \hat{\theta}) ~ \Rightarrow ~G(\pi,c,\theta)\geq G(\hat{\pi},c,\hat{\theta}),
\]
\begin{enumerate}[(i)]
	\item for $p\leq0$;
	\item for $p\in(0,1)$, if
	$
	0\leq c_t\leq 1, g^{\theta_t}(\pi_t), g^{\hat{\theta}_t}(\hat{\pi}_t)\geq0, \forall t\in[0,T].
	$
\end{enumerate}
\end{theorem}
\begin{proof}
	See Appendix \ref{app monotony and conclusion}.
\end{proof}

A partial order called kernel order is presented in Definition \ref{def monotonous}, under which we can compare elements in $\mathfrak{A}^\pi_0\times\Theta$. Based on the kernel order, Theorem \ref{g_monotony} reveals the monotonicity of $G$ over $\mathfrak{A}^\pi_0\times\Theta$. It is worth mentioning that the conditions in Theorem \ref{g_monotony} for $p\leq0$ and $p\in(0,1)$ are different because we use different techniques to overcome the difficulties encountered in the proof. Generally speaking, Assumption \ref{ass_CRRA} overcomes the difficulty of non-positive $p$, so we do not need to add more assumptions for the case of $p\leq0$. For $p\in(0,1)$, we overcome the difficulty by verifying the monotonicity with a stronger condition, that is, in a small domain of $G$ (see Eq. \eqref{rest of control} in Appendix \ref{app monotony and conclusion}). Fortunately, this small domain is large enough since we can verify that any optimal policy is in this domain, that is, the optimal policy satisfies the condition for $p\in(0,1)$ in Theorem \ref{g_monotony} (cf. the proof of Theorem \ref{exchange}).

Thanks to the monotonicity of global kernel, an immediate idea of constituting a saddle point of $G$ is to piece together the saddle points of local kernels from time $0$ to $T$. The existence of saddle points of local kernels has been studied by Neufeld and Nutz (2018), however, it is not sure whether the joint combination of saddle points of local kernels is measurable with respect to time. Following theorem answers this question mainly by using Kuratowski-Ryll-Nardzewski measurable selection theorem and measurable maximum theorem repeatedly (cf. Appendix \ref{app measurable selector}). Theorem \ref{measurable selector} is actually a measurable saddle point theorem.

\begin{theorem}\label{measurable selector}
There exist a $\pi^*\in\mathfrak{A}^\pi_0$ and a $\theta^*\in\Theta$ such that $(\pi^*_t, \theta^*_t)$ is a saddle point of $g$ on $ \mathcal{O}_t\times \Theta_t$ for every $t\in[0,T]$.
\end{theorem}
\begin{proof}
Using Sion's minimax theorem (cf. Theorem $4.2^\prime_{}$, Sion (1958)) and referring to the Section 3 of Neufeld and Nutz (2018), we have
\begin{equation}\nonumber
\inf_{\theta_t\in\Theta_t} \sup_{\pi_t\in\mathcal{O}_t}\{ g^{\theta_t}(\pi_t)\} = \sup_{\pi_t\in\mathcal{O}_t}\inf_{\theta_t\in\Theta_t} \{ g^{\theta_t}(\pi_t)\},
\end{equation}
due to the compactness of $\Theta_t$.
Therefore, the $\theta^*$ chosen in Lemma \ref{existence of theta} and the $\pi^*$ in Lemma \ref{existence of pi} satisfy that $\pi^*_t$ and $\theta^*_t$ constitute a saddle point of $g$ on $ \mathcal{O}_t\times \Theta_t$ for every $t\in[0,T]$. Details are in Appendix \ref{app measurable selector}.
\end{proof}

For using the Kuratowski-Ryll-Nardzewski measurable selection theorem to find a measurable saddle point, it is crucial to prove $\sup\limits_{x\in\mathcal{O}^n_t} g_t^{\cdot}(x)$ and $\min\limits_{(y,M,\mu)\in \Theta_t} g_t^{(y,M,\mu)}(\cdot)$ are continuous, which is provided by Lemmas \ref{continuous of I} and \ref{continuity of theta} in Appendix \ref{app measurable selector}. In more detail, Lemma \ref{continuous of I} holds thanks to the assumption that $\mathcal{C}$ is compact under $d_\mathcal{C}^{\varepsilon}$ with a positive $\varepsilon$. Actually, $\varepsilon>0$ is a sufficient and necessary condition in the proof of Lemma \ref{continuous of I}. When $\varepsilon=0$, $\mathcal{I}$ defined in Lemma \ref{continuous of I} is not H\"{o}lder continuous for any exponent $\alpha$ because
\[
\lim_{r\rightarrow 0}  \frac{| \cos(\vartheta(x,rz))^2 -  \cos(\vartheta(x,r\hat{z}))^2  |}{|rz - r\hat{z}|^\alpha } = \lim_{r\rightarrow 0}  \frac{1}{(\sqrt{2}r)^\alpha } = +\infty,
\]
where $z$ is the unit vector in the direction of $x$, $\hat{z}$ is a unit vector that is orthogonal to $x$, and $\vartheta(\cdot,\cdot)$ is the angle between two vectors.
After finding the candidates for the components of saddle point on $\mathfrak{A}^\pi_0\times \Theta$ in Theorem \ref{measurable selector}, we are going to find the maximizer of $G$ on $\mathfrak{A}^c_0$, which is a deterministic optimization and can be solved by dynamic programming.

\begin{theorem}\label{optimal c}
For any $\theta \in \Theta$ and $\pi\in \mathfrak{A}^\pi_0$, the maximizer of $G(\pi,\cdot,\theta)$ over $\mathfrak{A}^c_0$ is
\begin{small}
\begin{equation}\label{exp of c*}
c^*_t =\left\{
\begin{aligned}
&(T-t+1)^{-1}, &&p=0,\\
& \left\{\!\exp\!\left(\int_t^T\!\!\!\frac{p g^{\theta_u}(\pi_u)}{1-p}du\!\right)\!\!+\!\! \int_t^T\!\! \exp\!\left(\!\int_t^s  \frac{p g^{\theta_u}(\pi_u)}{1-p}du\!\right)ds\!\right\}^{\!-1\!}, && p\neq0.
\end{aligned}
\right.
\end{equation}
\end{small}
The maximum value $G(\pi,c^*,\theta)$ is
\begin{small}
\begin{equation}\label{optm G with c*}
\begin{aligned}
&\left\{
\begin{aligned}
&\!\int_0^T \!\bigg(\!\! \int_0^t g^{\theta_s}(\pi_s)  ds \!+\!\!g^{\theta_t}(\pi_t)\bigg)dt-(T+1)\log(T+1), &&p=0,\\
& \frac{1}{p}\!\left\{\!\exp\!\left(\int_0^T\!\!\!\frac{p g^{\theta_u}(\pi_u)}{1-p}du\!\right)\!\!+\!\! \int_0^T\!\! \exp\!\left(\int_0^s  \frac{p g^{\theta_u}(\pi_u)}{1-p}du\!\right)ds\!\right\}^{\!1\!-p}\!\!-\!\!\frac{1}{p}, && p\neq0.
\end{aligned}
\right.
\end{aligned}
\end{equation}
\end{small}
The range of maximizer $c^*$ is
\begin{equation}\label{range of c}
\left\{
\begin{aligned}
&\bigg[ \big(\frac{-pg_{\max}}{1-p} \vee 0\big) \wedge1, ~~\frac{-pg_{\min}}{1-p} \vee 1   \bigg],&&p>0,\\
&\bigg[ \big(\frac{-pg_{\min}}{1-p} \vee 0\big) \wedge1, ~~\frac{-pg_{\max}}{1-p} \vee 1  \bigg],&&p\leq0,
\end{aligned}
\right.
\end{equation}
where
\begin{equation}\nonumber
\begin{aligned}
&g_{\max} \triangleq \sup_{t\in[0,T]}g^{\theta_t}(\pi_t),
&g_{\min} \triangleq \inf_{t\in[0,T]}g^{\theta_t}(\pi_t).
\end{aligned}
\end{equation}
\end{theorem}
\begin{proof}
Lemma \ref{verify thm} in Appendix \ref{app opt c} gives an upper bound of $G(\pi,\cdot,\theta)$ over $\mathfrak{A}^c_0$, so we only need to verify that the $c^*$ defined by \eqref{exp of c*} attains the upper bound. Details are in Appendix \ref{app opt c}.
\end{proof}

Theorem \ref{optimal c} shows the maximum of $G(\pi,\cdot,\theta)$ on $\mathfrak{A}^c_0$ is attainable for any given $(\pi, \theta)\in \mathfrak{A}^\pi_0\times\Theta$. In the case of logarithmic utility, $c^*$ is a function increasing to one and dose not depend on other arguments.
For power utilities, the monotony of $c^*$ is unknown because the arguments $\pi$ and $\theta$ affect $c^*$ through the values of local kernels. Indeed,
the only sure thing is $c^*$ attains one at the terminal time, but it can be non-monotonic. This phenomenon appears in this robust problem because the confidence sets are time-varying. In the degenerated case that $\Theta$ is a constant correspondence, $c^*$ in \eqref{exp of c*} must be monotonically increasing or decreasing. For now, we are ready to constitute a saddle point of $G$ based on above three theorems.

\begin{theorem}\label{exchange}
There exist  $({\pi}^*, {c}^*)\in\mathfrak{A}_0$ and  $\theta^*\in\Theta$ constituting a saddle point of global kernel $G$ with following properties.
\begin{enumerate}[(i)]
\item For any $t\in[0,T]$, $(\pi^*_t, \theta^*_t)$ is a saddle point of local kernel $g$ on $ \mathcal{O}_t\times \Theta_t$.
\item For any $t\in[0,T]$,
\begin{equation}\nonumber
\quad c^*_t =\left\{
\begin{aligned}
&(T-t+1)^{-1}, &&p=0,\\
& \left\{\!\exp\!\left(\int_t^T\!\!\!\frac{p g^{\theta^*_u}(\pi^*_u)}{1-p}du\!\right)\!\!+\!\! \int_t^T\!\!\!\! \exp\!\left(\int_t^s \!\! \frac{p g^{\theta^*_u}(\pi^*_u)}{1-p}du\!\right)\!ds\!\right\}^{\!-1\!}, && p\neq0.
\end{aligned}
\right.
\end{equation}
\item The range of $c^*_t$ is
\begin{equation}\label{range of c*}
\begin{aligned}
\left\{
\begin{aligned}
&[0,~~1], &&p\geq0,\\
&\big[\frac{-p}{1-p}\min_{t\in[0,T]}\{g^{\theta^*_t}(\pi^*_t)\},~~1\big],&&p<0.
\end{aligned}
\right.
\end{aligned}
\end{equation}
\item The value of global kernel at saddle point is
\begin{equation}\nonumber
\quad\begin{aligned}
&\max_{(\pi,c)\in \mathfrak{A}_0} \min_{\theta \in {\Theta}} G((\pi,c),\theta)=\min_{\theta \in \Theta}\max_{(\pi,c)\in \mathfrak{A}_0} G((\pi,c),\theta)=G((\pi^*,c^*),\theta^*)\\
&=\left\{
\begin{aligned}
&\!\!\int_0^T \!\bigg(\!\! \int_0^t g^{\theta^*_s}(\pi^*_s)  ds \!+\!\!g^{\theta^*_t}(\pi^*_t)\bigg)dt-(T+1)\log(T+1), &&p\!=\!0,\\
& \frac{1}{p}\!\left(\!\exp\!\left(\int_0^T\!\!\!\frac{p g^{\theta^*_u}(\pi^*_u)}{1-p}du\!\right)\!\!+\!\! \int_0^T\!\! \exp\!\left(\!\int_0^s  \frac{p g^{\theta^*_u}(\pi^*_u)}{1-p}du\!\right)ds\!\right)^{\!1\!-\!p}\!\!-\!\!\frac{1}{p}, && p\!\neq\!0.
\end{aligned}
\right.
\end{aligned}
\end{equation}
\end{enumerate}
\end{theorem}

\begin{proof}
We only need to verify that the $\pi^*$ and $\theta^*$ selected in Theorem \ref{measurable selector} and the $c^*$ defined by condition $(ii)$ can constitute a saddle point of $G$.
Details are in Appendix \ref{app monotony and conclusion}, where Theorems \ref{g_monotony} and \ref{optimal c} play important roles.
\end{proof}

Recalling the range \eqref{range of c} in Theorem \ref{optimal c}, we notice that the maximizer of $G(\pi,\cdot,\theta)$ over $\mathfrak{A}^c_0$ can be greater than one, which means people may consume more than the total wealth instantaneously (over-consumption occurs). However, the optimal consumption ratio provided by Theorem \ref{exchange} is not more than one according to \eqref{range of c*}.
The over-consumption will not happen under the worst-case model because
\begin{equation}\nonumber
\begin{aligned}
& \frac{-p}{1-p} \min_{t\in[0,T]}\{g^{\theta^*_t}(\pi^*_t)\}\leq1 \text{ for } p>0,
 &   \frac{-p }{1-p} \max_{t\in[0,T]}\{g^{\theta^*_t}(\pi^*_t)\}\leq 1  \text{ for } p\leq0
\end{aligned}
\end{equation}
for $\pi^*$ and $\theta^*$ given by Theorem \ref{exchange}.
Actually, the first condition holds thanks to the optimality of saddle point, that is, for any $p>0$ and $t\in[0,T]$
\[
g_t^{\theta^*_t}(\pi^*_t)\geq g_t^{\theta^*_t}(0)=0\geq \frac{1-p}{-p}.
\]
The second one is tenable because Assumption \ref{ass_CRRA} limits the Sharpe ratio of the market while the over-consumption only occurs in a good enough market.

\subsection{Kernel of CARA}
We try to find a saddle point of $H$ under Assumptions \ref{ass_2.1}, \ref{ass_2.2}, and \ref{ass_CARA} by analogous procedures as above. We denote
\begin{equation}\nonumber
\begin{aligned}
\mathfrak{A}^\Pi_0 \triangleq \big\{\Pi: [0,T] \rightarrow \R^d  \big|
\text{ measurable} \big\},~
\mathfrak{A}_0^D  \triangleq \big\{D : [0,T] \rightarrow \R \big|
\text{ measurable}\big\},
\end{aligned}
\end{equation}
then $\mathfrak{A}_0 =\mathfrak{A}^\Pi_0\times\mathfrak{A}_0^D$ and $H$ can be regarded as a function with three arguments.
According to the comparison between global kernels of CARA and CRRA utilities in Section 4, the procedures of analyzing $H$ should be similar to $G$ with negative $p$.
So we omit the proofs in the following theorems and simply state the main ideas. Besides, it will be easier to understand them while comparing them with the conclusions about $G$.

\begin{Def}\label{def monotonous CARA}
	For any $(\Pi, \theta),~(\hat{\Pi}, \hat{\theta})\in \mathfrak{A}^\Pi_0\times \Theta$, we say $(\Pi, \theta) \succcurlyeq(\hat{\Pi}, \hat{\theta})$ if and only if
	\[
	h_t^{\theta_t}(\Pi_t) \geq h_t^{\hat{\theta}_t}(\hat{\Pi}_t), \forall t\in[0,T].
	\]
	We call $(\Pi, \theta)$ is greater than $(\hat{\Pi}, \hat{\theta})$ in the kernel order.
\end{Def}

\begin{theorem}\label{g_monotony CARA}
	For any $D \in \mathfrak{A}_0^D$, and $(\Pi, \theta),~(\hat{\Pi}, \hat{\theta})\in \mathfrak{A}^\Pi_0\times \Theta$,
	\[
	(\Pi, \theta)\succcurlyeq(\hat{\Pi}, \hat{\theta})
	\Rightarrow H(\Pi,D,\theta)\geq H(\hat{\Pi},D,\hat{\theta}).
	\]
\end{theorem}
\begin{proof}
The proof can be easily finished once referring to the proof of Theorem \ref{g_monotony} for $p<0$. Assumption \ref{ass_CARA} ensures $q_t(h_t -D_t)+U(D_t) \leq 0$ and plays the same role as Assumption \ref{ass_CRRA} in Theorem \ref{g_monotony}.
\end{proof}

\begin{theorem}\label{measurable selector CARA}
There exist a $\Pi^*\in\mathfrak{A}^\Pi_0$ and a $\theta^*\in\Theta$ such that $(\Pi^*_t, \theta_t)$ is a saddle point of $h_t$ over $ \R^d\times \Theta_t$ for every $t\in[0,T]$.
\end{theorem}
\begin{proof}
The main tools are Kuratowski-Ryll-Nardzewski measurable selection theorem and measurable maximum theorem as in the proof of Theorem \ref{measurable selector}.
A subset of $\R^d$ should be proposed as
$\mathcal{O}_t^n  \triangleq \left\{x\in\R^d : |x|\leq n\right\}$
to avoid the singularity of local kernel $h_t$ for any $t$.
\end{proof}

Thanks to the monotonicity of $H$ described in Theorem \ref{g_monotony CARA}, the $\Pi^*$ and $\theta^*$ selected in Theorem \ref{measurable selector CARA} are the candidates for the components of $H$'s saddle point. Then we study the maximizer of  $H$ over $\mathfrak{A}^D_0$ for any given $\Pi$ and $\theta$.

\begin{theorem}\label{optimal c  CARA}
For any $\theta \in \Theta$ and $\Pi\in \mathfrak{A}^\Pi_0$, the maximizer of $H(\Pi,\cdot,\theta)$ over $\mathfrak{A}_0^D$ is
\begin{equation}\label{exp of c* CARA}
D^*_t = q_t \int_t^T h_s^{\theta_s}(\Pi_s) ds, ~\forall t\in[0,T].
\end{equation}
The maximum value is
\begin{equation}\label{optm G with c* CARA}
\begin{aligned}
&H(\Pi,D^*,\theta)=  \frac{1}{a}\left(1-(T+1) \exp\left(-\frac{a}{T\!+\!1}\int_0^T h_s^{\theta_s}(\Pi_s) ds\right)\right).
\end{aligned}
\end{equation}%
\end{theorem}
\begin{proof}
Lemma \ref{verify thm CARA} in Appendix \ref{app opt c} gives an upper bound of $H(\Pi,\cdot,\theta)$ over $\mathfrak{A}_0^D$, so it is enough to verify $D^*$ defined by \eqref{exp of c* CARA} attains the upper bound. More details are in Appendix \ref{app opt c}.
\end{proof}

It not difficult to realize the procedure of proving Theorem \ref{optimal c  CARA} is analogous to Theorem \ref{optimal c}, but we still show the details in Appendix \ref{app opt c} because the  associated HJB equations have different forms.

\begin{theorem}\label{exchange CARA}
There exist $({\Pi}^*, {D}^*)\in\mathfrak{A}_0$ and $\theta^*\in\Theta$ constituting a saddle point of global kernel $H$ with following properties.
\begin{enumerate}[(i)]
\item For any $t\in[0,T]$, $(\Pi^*_t, \theta^*_t)$ is a saddle point of local kernel $h_t$  over $\R^d\times \Theta_t$.
\item For any $t\in[0,T]$,
\begin{equation}\nonumber
D^*_t = q_t\int_t^T h_s^{\theta^*_s}(\Pi^*_s) ds.
\end{equation}
\item The value of global kernel at saddle point is
\begin{equation}\nonumber
\begin{aligned}
&
\min_{\theta \in \Theta}\max_{(\Pi,D)\in \mathfrak{A}_0} H((\Pi,D),\theta)
=\max_{(\Pi,D)\in \mathfrak{A}_0} \min_{\theta \in \Theta}H((\Pi,D),\theta)
=H((\Pi^*,D^*),\theta^*)\\
&=\frac{1}{a}\left(1-(T+1) \exp\left(-\frac{a}{T+1}\int_0^T h_s^{\theta^*_s}(\Pi^*_s) ds\right)\right).
\end{aligned}
\end{equation}
\end{enumerate}

\end{theorem}
\begin{proof}
The main proof is analogous to the proof of Theorem \ref{exchange} for $p<0$, while using Theorems \ref{g_monotony CARA}, \ref{measurable selector CARA}, and \ref{exp of c* CARA}.
\end{proof}
In Definition \ref{def g_CARA}, the function $h_t$ depends on the time $t$, thus $\Pi^*$ is time-dependent even though the confidence set $\Theta_t$ keeps constant. This is an explanation of why CARA investors are not myopic from a mathematical point of view.
By the optimality of saddle point, we have
\[
h_t^{\theta^*_t}(\Pi^*_t)\geq h_t^{\theta^*_t}(0)=0,~\forall t\in[0,T],
\]
which means $D^*$ is a non-negative function.

\vskip 15pt
\setcounter{equation}{0}
\section{Conclusion}
The main contribution of this paper is solving the robust consumption-investment problem in the market with jumps by martingale method. Conceptually, the uncertainty set consists of all the semimartingale measures, under which the DC triplet of the log-price processes is a PII triplet (measurable selector from the correspondence $\Theta$) almost surely. Under this framework, the  instantaneous confidence sets of all possible L\'{e}vy triplets are time-varying, which causes mensurability problems.

Thanks to the well-defined policy for each utility, the martingale equality describes the relationship between the objective function and global kernel, which is the core in the deterministic-to-stochastic paradigm. Thus the robust consumption-investment problem can be solved by finding a saddle point of the global kernel.
However, it is difficult to find a measurable saddle point of the global kernel, because the confidence set $\Theta_t$ is time-varying and the global kernel is a functional.
Fortunately, while regarding the global kernel as a function with three arguments --- an investment policy, a consumption policy, and a DC triplet, we obtain two essential conclusions .
For fixed consumption policy, the global kernel is an increasing functional under the kernel order; for fixed investment policy and DC triplet, an optimal consumption policy can be explicitly expressed. Synthesizing these two conclusions, a saddle point of the global kernel can be found by the measurable saddle point theorem.

For a CRRA investor, the choices of optimal investment policy and worst-case differential characteristics are both myopic, uniquely determined by the instantaneous local kernel, i.e., the instantaneous model of the market. Thus, the optimal portfolio is changeless if the confidence set $\Theta_t$ is constant. However, the time effect appears for CARA investors. The optimal investment policy and the worst-case differential characteristics are influenced by the remaining duration as well as the instantaneous model. Besides, the optimal consumption policy is non-myopic for any utilities and is affected by both local kernels and remaining duration.

\vskip 10pt
\vskip 10pt
\appendix
\vskip 10pt
\section{Proofs of Lemma \ref{J=G}, Theorem \ref{thm_CRRA}, and Lemma \ref{J=G_CARA}}\label{app1}
\textit{\textbf{Proof of Lemma \ref{J=G}}.}
By (\ref{sde of X}), (\ref{sde of W_pic}), and It\^{o}'s formula for jump processes (cf. {\O}ksendal and Sulem (2005)), we have
\begin{equation}\label{exp of logW}
\begin{aligned}
\log(W^{(\bar{\pi},\bar{c})}_t)
=& \log(w_0) \!+\! \int_0^t \!(g^{\theta^P_s}(\bar{\pi}_s) \!-\!\bar{c}_s )ds  \!+\! \int_0^t \bar{\pi}^T_s\sigma^P_s dB_s\\
&+\int_0^t\!\!\int_{\R^d}\log(1+\bar{\pi}_s^T z) \tilde{\nu}^P(dz,ds)
\end{aligned}
\end{equation}
for any $ t\in[0,T]$.
Take expectation on both sides of \eqref{exp of logW}, then
\begin{equation}\nonumber
\begin{aligned}
\E^P[\log(W^{(\bar{\pi},\bar{c})}_t)]=\log(w_0) + \E^P[\int_0^t (g^{\theta^P_s}(\bar{\pi}_s) -\bar{c}_s )ds],
\end{aligned}
\end{equation}
and
\begin{small}
\begin{equation}\nonumber
\begin{aligned}
&J((\bar{\pi}, \bar{c}),P) \\
=&\E^P\bigg[\int_0^T \log(\bar{c}_tW^{(\bar{\pi},\bar{c})}_t) dt +\log(W^{(\bar{\pi},\bar{c})}_T)\bigg]\\
=&\E^P\bigg[\int_0^T \bigg(\log(\bar{c}_t)+ \log(W^{(\bar{\pi},\bar{c})}_t)\bigg) dt +\log(W^{(\bar{\pi},\bar{c})}_T)\bigg]\\
=&\E^P\bigg[\int_0^T\!\! \bigg(\!\log(\bar{c}_t)\!+\! \log(w_0) \!+\!\! \int_0^t\! (g^{\theta^P_s}(\bar{\pi}_s) \!-\!\bar{c}_s )ds\bigg) dt
+\! \log(w_0)  \!\!+\!\! \int_0^T\! \!\! (g^{\theta^P_t}(\bar{\pi}_t)\!-\!\bar{c}_t)dt \bigg]\\
=&(T+1)\log(w_0) +¡¡\int_\Omega  G((\bar{\pi},\bar{c}),\theta^P) P(d\omega).
\end{aligned}
\end{equation}
\end{small}
This is the result for logarithmic utility ($p=0$). Next we use the expansion of $\log(W_t)$ in \eqref{exp of logW} to expand $W_t^p$ as a product. For any $ t\in[0,T]$,

\begin{equation}\label{exp of W^p}
\begin{aligned}
&(W_t^{(\bar{\pi},\bar{c})})^p\\
=&\exp(p\log(W^{(\bar{\pi},\bar{c})}_t))\\
=& w_0^p \exp\!\!\bigg\{ \!\int_0^t\!\! p(g^{\theta^P_s}(\bar{\pi}_s) \!-\!\bar{c}_s )ds \!+\! \int_0^t \!\! p\bar{\pi}_s^T\sigma^P_s dB_s\\
&+\! \int_0^t\!\!\!\int_{\R^d}\!p\log(1\!+\!\bar{\pi}_s^T z) \tilde{\nu}^P(dz,ds)\!\bigg\}\\
=&w_0^p \frac{dQ^{(\bar{\pi},\theta^P)}}{dP}(t) \exp\bigg( \int_0^t p(g^{\theta^P_s}(\bar{\pi}_s)-\bar{c}_s) ds\bigg).
\end{aligned}
\end{equation}
Substitute (\ref{exp of W^p}) into the objective function and we obtain the final result:
\begin{equation}\nonumber
\begin{aligned}
&J((\bar{\pi}, \bar{c}),P) &&\\
=&\frac{1}{p} \E\bigg[\int_0^T \bar{c}_t^p (W_t^{(\bar{\pi},\bar{c})})^p  dt + (W_T^{(\bar{\pi},\bar{c})})^p \bigg]&&\\
=&\frac{w_0^p}{p} \E\bigg[\int_0^T \bar{c}_t^p \frac{dQ^{(\bar{\pi},\theta^P)}}{dP}(t) \exp\bigg(\! \int_0^t\!\! p(g^{\theta^P_s}(\bar{\pi}_s)  \!-\!\bar{c}_s) ds\bigg)   dt &&\\
&+ \frac{dQ^{(\bar{\pi},\theta^P)}}{dP}(T) \exp\bigg( \!\int_0^T\!\! p(g^{\theta^P_s}(\bar{\pi}_s)  \!-\!\bar{c}_s) ds\bigg)\bigg]&&\\
=&\frac{w_0^p}{p} \E^{Q^{(\bar{\pi},\theta^P)}}\bigg[\int_0^T\!\! \bar{c}_t^p  \exp\bigg(\! \int_0^t p(g^{\theta^P_s}(\bar{\pi}_s) \! -\!\bar{c}_s) ds\bigg)   dt &&\\
&+ \exp\bigg( \int_0^T\!\! p(g^{\theta^P_s}(\bar{\pi}_s)  \!-\!\bar{c}_s) ds\bigg)\bigg]&&\\
=&\frac{w_0^p}{p} \E^{Q^{(\bar{\pi},\theta^P)}}\bigg[\int_0^T\!\! \bar{c}_t^p  \exp\bigg(\! \int_0^t \!\! p(g^{\theta^P_s}(\bar{\pi}_s) \!-\!\bar{c}_s) ds\bigg)   dt &&\\
&+ 1+\int_0^T \!\!\exp\bigg(\! \int_0^t\!\! p(g^{\theta^P_s}(\bar{\pi}_s)\!-\!\bar{c}_s) ds\bigg)p(g^{\theta^P_t}(\bar{\pi}_t)  \!-\!\bar{c}_t )dt\bigg]&&\\
=&w_0^p \bigg(\frac{1}{p}+
\int_\Omega  G((\bar{\pi},\bar{c}),\theta^P) Q^{(\bar{\pi},\theta^P)}(d\omega)\bigg). &&\square
\end{aligned}
\end{equation}

\vskip 10pt
\textit{\textbf{Proof of Theorem \ref{thm_CRRA}}.}
$(({\pi}^*, {c}^*), \theta^*)$ is a saddle point of $G$ by Theorem \ref{exchange}, so we mainly verify the $((\tilde{\pi}^*, \tilde{c}^*), P^*)$ defined by conditions $(ii)$ and $(iii)$ is a saddle point of $J$.
We only show the details for power utility here, because the proof for power utility is more general than the logarithmic utility.

For any $(\bar{\pi},\bar{c})\in \mathfrak{A}$, using Lemma \ref{J=G}, we have
\begin{equation}\label{log_1}
\begin{aligned}
&\inf_{P\in\mathfrak{P} } J((\bar{\pi}, \bar{c}),P) \\
=& \frac{w_0^p}{p} + w_0^p\inf_{P\in\mathfrak{P} }
\int_\Omega  G((\bar{\pi}(\omega),\bar{c}(\omega)),\theta^P(\omega)) Q^{(\bar{\pi},\theta^P)}(d\omega)\\
\leq & \frac{w_0^p}{p} +w_0^p\inf_{P\in\mathfrak{P} }
\int_\Omega  \sup_{(\pi,c)\in\mathfrak{A}_0}
G((\pi,c),\theta^P(\omega)) Q^{(\pi,\theta^P)}(d\omega)\\
\leq & \frac{w_0^p}{p} +w_0^p\inf_{P\in\mathfrak{P}_0 }\int_\Omega  \sup_{(\pi,c)\in\mathfrak{A}_0}
G((\pi,c),\theta^P) Q^{(\pi,\theta^P)}(d\omega) \\
= & \frac{w_0^p}{p} +w_0^p\inf_{P\in\mathfrak{P}_0 } \sup_{(\pi,c)\in\mathfrak{A}_0}
G((\pi,c),\theta^P)\\
=& \frac{w_0^p}{p} +w_0^p\inf_{\theta \in\Theta } \sup_{(\pi,c)\in\mathfrak{A}_0}G((\pi,c),\theta)\\
=& \frac{w_0^p}{p} +w_0^pG((\pi^*,c^*),\theta^*),
\end{aligned}
\end{equation}
where the second equality holds because $\theta^P$ is independent of the probability space for a possible PII measure in $\mathfrak{P}_0$.
Because $\inf_{\theta \in\Theta } G((\pi^*,c^*),\theta)$ is constant and Lemma \ref{J=G} holds,
\begin{equation}\label{log_2}
\begin{aligned}
& \frac{w_0^p}{p} +w_0^p\inf_{\theta \in\Theta } G((\pi^*,c^*),\theta) \\
=& \frac{w_0^p}{p} +w_0^p\inf_{P\in\mathfrak{P} } \int_\Omega \inf_{\theta \in\Theta } G((\pi^*,c^*),\theta)Q^{(\tilde{\pi}^*,\theta^P)}(d\omega) \\
\leq & \frac{w_0^p}{p} +w_0^p\inf_{P\in\mathfrak{P} } \int_\Omega  G((\pi^*,c^*),\theta^P(\omega)) Q^{(\tilde{\pi}^*,\theta^P)}(d\omega) \\
=&\inf_{P\in\mathfrak{P} } J((\tilde{\pi}^*, \tilde{c}^*),P),
\end{aligned}
\end{equation}
where $(\tilde{\pi}^*, \tilde{c}^*)$ is the generated policy from $(\pi^*, c^*)$ according to Definition
\ref{def genertae strategy}. By (\ref{log_1}), (\ref{log_2}), noting $(\bar{\pi},\bar{c})$ is arbitrary and $((\pi^*,c^*),\theta^*)$ is $G$'s saddle point, we conclude that
\begin{equation}\nonumber
\begin{aligned}
&\sup_{(\bar{\pi},\bar{c})\in \mathfrak{A}}\inf_{P\in\mathfrak{P} } J((\bar{\pi}, \bar{c}),P)
\leq\frac{w_0^p}{p} +w_0^p G((\pi^*,c^*),\theta^*) \\
&\leq\inf_{P\in\mathfrak{P} } J((\tilde{\pi}^*, \tilde{c}^*),P) \leq \sup_{(\bar{\pi},\bar{c})\in \mathfrak{A}}\inf_{P\in\mathfrak{P} }  J((\bar{\pi}, \bar{c}),P),\\
\Rightarrow&\sup_{(\bar{\pi},\bar{c})\in \mathfrak{A}}\inf_{P\in\mathfrak{P} } J((\bar{\pi}, \bar{c}),P)
=\frac{w_0^p}{p} +w_0^p G((\pi^*,c^*),\theta^*).
\end{aligned}
\end{equation}

On the other hand,
\begin{equation}\label{log_3}
\begin{aligned}
&\frac{w_0^p}{p} +w_0^pG((\pi^*,c^*),\theta^*) \\
=&\frac{w_0^p}{p} +w_0^p\inf_{\theta \in\Theta } \sup_{(\pi,c)\in \mathfrak{A}_0}G((\pi,c),\theta) \\
=&\frac{w_0^p}{p} + w_0^p\inf_{P\in\mathfrak{P}_0  } \sup_{(\pi,c)\in \mathfrak{A}_0}G((\pi,c),\theta^P) \\
=&\frac{w_0^p}{p} + w_0^p\inf_{P\in\mathfrak{P}_0 } \sup_{(\bar{\pi},\bar{c})\in \mathfrak{A}}  \int_\Omega \sup_{(\pi,c)\in \mathfrak{A}_0}G((\pi,c),\theta^P) Q^{(\bar{\pi},\theta^P)}(d\omega)\\
\geq &\frac{w_0^p}{p} + w_0^p\inf_{P\in\mathfrak{P}_0 } \sup_{(\bar{\pi},\bar{c})\in \mathfrak{A}} \int_\Omega G((\bar{\pi}(\omega),\bar{c}(\omega)),\theta^P) Q^{(\bar{\pi},\theta^P)}(d\omega)\\
\geq &\frac{w_0^p}{p} + w_0^p\inf_{P\in\mathfrak{P} } \sup_{(\bar{\pi},\bar{c})\in \mathfrak{A}} \int_\Omega G((\bar{\pi}(\omega),\bar{c}(\omega)),\theta^P(\omega)) Q^{(\bar{\pi},\theta^P)}(d\omega)\\
=&\inf_{P\in\mathfrak{P} } \sup_{(\bar{\pi},\bar{c})\in \mathfrak{A}} J((\bar{\pi}, \bar{c}),P).
\end{aligned}
\end{equation}
The second equality holds because there is a bijection between $\mathfrak{P}_0$ and $\Theta$. The third holds because $\sup_{(\pi,c)\in \mathfrak{A}_0}G((\pi,c),\theta^P)$ is independent of the probability space for any possible PII measure $P\in\mathfrak{P}_0$. The last equality comes from Lemma \ref{J=G}. The first inequality holds because $(\bar{\pi}(\omega),\bar{c}(\omega))\in\mathfrak{A}_0$ almost surely. The second inequality holds since $\mathfrak{P}$ contains $\mathfrak{P}_0$.
By (\ref{log_2}), (\ref{log_3}), and the property of saddle point, we obtain
\begin{equation}\nonumber
\qquad\qquad
\begin{aligned}
&\inf_{P\in\mathfrak{P} } \sup_{(\bar{\pi},\bar{c})\in \mathfrak{A}} J((\bar{\pi}, \bar{c}),P)
\leq\frac{w_0^p}{p} + w_0^p G((\pi^*,c^*),\theta^*) \\
&\leq \inf_{P\in\mathfrak{P} }  J((\tilde{\pi}^*, \tilde{c}^*),P)
\leq\inf_{P\in\mathfrak{P} } \sup_{(\bar{\pi},\bar{c})\in \mathfrak{A}} J((\bar{\pi}, \bar{c}),P),\\
\Rightarrow& \inf_{P\in\mathfrak{P} } \sup_{(\bar{\pi},\bar{c})\in \mathfrak{A}} J((\bar{\pi}, \bar{c}),P)
=\frac{w_0^p}{p} + w_0^p G((\pi^*,c^*),\theta^*).  \qquad\qquad\square
\end{aligned}
\end{equation}

\vskip 10pt
\textit{\textbf{Proof of Lemma \ref{J=G_CARA}}}.
By (\ref{sde of X}), (\ref{sde of W_PiD}), and It\^{o}'s formula for jump diffusion, $q_tW_t$ has an expression in form of integral.
\begin{equation}\nonumber
\begin{aligned}
&q_tW^{(\bar{\Pi},\bar{D})}_t = e^{\int_t^T-q_sds}W^{(\bar{\Pi},\bar{D})}_t
= q_0w_0 + \int_0^t (q_s\bar{\Pi}^T_sb^P_s - q_s\bar{D}_s ) ds \\
&+ \int_0^t q_s\bar{\Pi}^T_s \sigma^P_s dB_s +\int_0^t\int_{\R^d}q_s \bar{\Pi}_s^T z \tilde{\nu}^P(dz,ds).
\end{aligned}
\end{equation}
We set $q_t \!=\! (T-t+1)^{-1}$ for any $t\in[0,T]$ because it satisfies $q_t= e^{\int_t^T-q_sds}$ and ensures above equality. Then,
\begin{equation}\nonumber
\begin{aligned}
&\exp(-aq_tW^{(\bar{\Pi},\bar{D})}_t) \\
=&e^{-aq_0w_0} \exp\!\bigg\{\int_0^t (-aq_s\bar{\Pi}^T_sb^P_s +a q_s\bar{D}_s) ds - \int_0^t aq_s\bar{\Pi}^T_s \sigma^P_s dB_s \\
&-\int_0^t\int_{\R^d}aq_s \bar{\Pi}_s^T z \tilde{\nu}^P(dz,ds)\bigg\}\\
=&e^{-aq_0w_0} \exp\!\bigg\{ \!-\!\int_0^t aq_s\bar{\Pi}^T_s \sigma^P_s dB_s - \frac{1}{2}\int_0^t a^2q^2_s\bar{\Pi}^T_s \Sigma^P_s \bar{\Pi}_s ds \bigg\}\\
&\times\exp\!\bigg\{\!\!\int_0^t\!\!\int_{\R^d}\!-aq_s \bar{\Pi}_s^T z \nu^P(dz,ds)\!+\! \int_0^t\!\!\int_{\R^d}\!(1\!-\!e^{\!-aq_s \bar{\Pi}_s^T z}) F^P_t(dz)ds\!\bigg\} \\
&\times\exp\!\bigg\{\int_0^t \bigg(-aq_s\bar{\Pi}^T_sb^P_s +a q_s\bar{D}_s  + \frac{1}{2} a^2q^2_s\bar{\Pi}^T_s \Sigma^P_s \bar{\Pi}_s \\
&+ \int_{\R^d} (aq_s \bar{\Pi}_s^T z -1+e^{-aq_s \bar{\Pi}_s^T z}) F^P_t(dz) \bigg)ds\bigg\}\\
=&e^{-aq_0w_0}\frac{dQ^{(\bar{\Pi},\theta^P)}}{dP}(t)\exp\!\bigg(\int_0^t -aq_sh^{\theta^P_s}_s(\bar{\Pi}_s) +a q_s\bar{D}_s ds\bigg),
\end{aligned}
\end{equation}
and
\begin{equation}\nonumber
\begin{aligned}
&J((\bar{\Pi}, \bar{D}),P) \\
=&\E^P\bigg[\!\int_0^T \!\!\frac{\exp(\!-a (q_tW^{(\bar{\Pi},\bar{D})}_t+ \bar{D}_t))}{-a}dt \!+\!\frac{\exp(-a W^{(\bar{\Pi},\bar{D})}_T)}{-a}\!\bigg]&\\
=&\frac{1}{-a} \E^P\bigg[\int_0^T e^{-a\bar{D}_t}\exp(-a q_t W^{(\bar{\Pi},\bar{D})}_t)dt +\exp(-a q_T W^{(\bar{\Pi},\bar{D})}_T)\bigg]&\\
=&\frac{e^{-aq_0w_0}}{-a}\E^{Q^{(\bar{\Pi},\theta^P)}}\!\bigg[\!\!\int_0^T\!\! e^{-a\bar{D}_t}\exp\!\bigg(\!\int_0^t \!(-aq_sh^{\theta^P_s}_s(\bar{\Pi}_s) \!+\!a q_s\bar{D}_s ) ds\!\bigg)dt&\\
&+\exp\!\bigg(\int_0^T \!(-aq_sh^{\theta^P_s}_s(\bar{\Pi}_s) +a q_s\bar{D}_s) ds\bigg)\bigg]&\\
=&\frac{e^{-aq_0w_0}}{-a}\E^{Q^{(\bar{\Pi},\theta^P)}}\bigg[1+\int_0^T \exp\!\bigg(\int_0^t -aq_sh^{\theta^P_s}_s(\bar{\Pi}_s) +a q_s\bar{D}_s ds\bigg)&\\
&\times\bigg(e^{-a\bar{D}_t}-aq_th^{\theta^P_t}_t(\bar{\Pi}_t) +a q_t\bar{D}_t\bigg)  dt\bigg] &\\
=&e^{-aq_0w_0}\bigg(\frac{-1}{a}+ \int_{\Omega} H((\bar{\Pi},\bar{D}),\theta^P) Q^{(\bar{\Pi},\theta^P)}(d\omega)\bigg).
\end{aligned}
\end{equation}
$\hfill\square$

\vskip 10pt
\section{ Proofs of Theorems \ref{g_monotony} and \ref{exchange}}\label{app monotony and conclusion}
\vskip 10pt
\textit{\textbf{Proof of Theorem \ref{g_monotony}}.}
The proof is divided into three cases: $p=0$, $p<0$ and $p\in(0,1)$.
\vskip 5pt
\textit{Case 1: $p=0$.}
For logarithmic utility, $g^{\theta_t}(\pi_t)$ and $c_t$ are separate in the expression of $G$, i.e.,
\begin{equation}\nonumber
\begin{aligned}
G(\pi,c,\theta)
=&\!\!\int_0^T \!\bigg(\!\! \int_0^t (g^{\theta_s}(\pi_s) \!-\!c_s) ds \!+\!\!g^{\theta_t}(\pi_t) \!-\!c_t+U(c_t)\bigg)dt\\
=&\!\!\int_0^T \!\!\!\bigg(\!\!\int_0^t g^{\theta_s}(\pi_s) ds+\!\!g^{\theta_t}(\pi_t)\bigg)dt+ \!\!\int_0^T \!\!\!\bigg(\!\!\! -\!\! \int_0^t \!c_s ds  \!-\!c_t+U(c_t)\bigg)dt,
\end{aligned}
\end{equation}
thus,
\begin{equation}\nonumber
\begin{aligned}
&(\pi, \theta) \succcurlyeq (\hat{\pi}, \hat{\theta})\\
\Rightarrow & \int_0^t g^{\theta_s}(\pi_s) ds+\!\!g^{\theta_t}(\pi_t)\geq \int_0^t g^{\hat{\theta}_s}(\hat{\pi}_s) ds+\!\!g^{\hat{\theta}_t}(\hat{\pi}_t),~\forall t\in[0,T] \\
\Rightarrow& G(\pi,c,\theta)\geq G(\hat{\pi},c,\hat{\theta}).
\end{aligned}
\end{equation}
\vskip 5pt
\textit{Case 2: $p<0$.} For any $t\in[0,T]$, let $\bar{g}_t \triangleq g^{\theta_t}(\pi_t)-g^{\hat{\theta}_t}(\hat{\pi}_t)$, thus $\bar{g}_t\geq0$ according to $(\pi, \theta) \succcurlyeq (\hat{\pi}, \hat{\theta})$. We use variational method here. For any $\delta\in[0,1]$, let
\begin{equation}\nonumber
\begin{aligned}
&g^\delta_t  \triangleq g^{\hat{\theta}_t}(\hat{\pi}_t) + \delta \bar{g}_t, ~\forall t\in[0,T], \\
&f(\delta) \triangleq  \exp\bigg( \!\!\int_0^t \!p(g^\delta_s \!-\!c_s) ds\!\bigg) \bigg(\!g^\delta_t \!-\!c_t\!+\!U(c_t)\! \bigg).
\end{aligned}
\end{equation}
The derivative of $f$ is
\begin{equation}\nonumber
\begin{aligned}
\frac{d}{d\delta}f(\delta)
=\exp\bigg( \!\!\int_0^t \!p(g^\delta_s \!-\!c_s) ds\!\bigg) \bigg((g^\delta_t \!-\!c_t\!+\!U(c_t))\int_0^t \!p\bar{g}_s ds+\bar{g}_t\bigg).
\end{aligned}
\end{equation}
Using Assumption \ref{ass_CRRA} and noticing $U(1+ x)- U(1)- x\leq0~(\forall x>-1)$, we know for any $p<0$ and $t\in[0,T]$,
\begin{equation}\label{est_g_orig}
\begin{aligned}
&g^{\theta_t}(\pi_t)
\leq &\pi_t^T b_t \!-\! \frac{1\!-\!p}{2}\pi_t^T \Sigma_t \pi_t
\leq &\frac{b_t^T(\Sigma_t)^{-1}b_t}{2(1-p)}\leq &\frac{1-p}{-p}.
\end{aligned}
\end{equation}
Above estimation still holds for $g^{\hat{\theta}_t}(\hat{\pi}_t)$ as well as $g^\delta_t$. By $U(c_t)-c_t\leq \frac{1}{p}-1$, we get
\begin{equation}\nonumber
\begin{aligned}
&g^\delta_t -c_t+U(c_t) \leq 0,
\end{aligned}
\end{equation}
which leads to $\frac{d}{d\delta}f(\delta)\geq0$ directly. Thus $f(1)\geq f(0)$, i.e., $G(\pi,c,\theta)\geq G(\hat{\pi},c,\hat{\theta})$.
\vskip 5pt
\textit{Case 3: $p\in(0,1)$.}
We use the same notations as in case 2, while  more constraints are given as follows:
\begin{equation}\nonumber
\begin{aligned}
&0\leq c_t\leq 1, ~g^{\theta_t}(\pi_t), ~g^{\hat{\theta}_t}(\hat{\pi}_t)\geq0, ~\forall t\in[0,T].\\
\end{aligned}
\end{equation}
Because $g^\delta_t\geq0$ and $U(c_t)-c_t\geq0$ for any $t\in[0,T]$, we have
\begin{equation}\nonumber
\begin{aligned}
&g^\delta_t -c_t+U(c_t)\geq0
\Rightarrow ~& \frac{d}{d\delta}f(\delta)\geq0
\Rightarrow~ & G(\pi,c,\theta)\geq G(\hat{\pi},c,\hat{\theta}).
\end{aligned}
\end{equation}
$\hfill\square$

~\vskip 10pt
\textit{\textbf{Proof of Theorem \ref{exchange}}.}
It is sufficient to prove
\[
\min_{\theta \in \Theta}\sup_{\pi\in \mathfrak{A}^\pi_0}\sup_{c\in \mathfrak{A}^c_0} G(\pi,c,\theta) =\sup_{c\in \mathfrak{A}^c_0}\sup_{\pi\in \mathfrak{A}^\pi_0} \min_{\theta \in \Theta}G(\pi,c,\theta)= G(\pi^*,c^*,\theta^*),
\]
where $\pi^*$ and $\theta^*$ are selected in Theorem \ref{measurable selector} and $c^*$ is defined by condition $(ii)$ in Theorem \ref{exchange}. To do this, we calculate
\[
\min_{\theta \in \Theta}\sup_{\pi\in \mathfrak{A}^\pi_0}\sup_{c\in \mathfrak{A}^c_0} G(\pi,c,\theta) \quad\text{   and   }\quad \sup_{c\in \mathfrak{A}^c_0}\sup_{\pi\in \mathfrak{A}^\pi_0} \min_{\theta \in \Theta}G(\pi,c,\theta),
\]
respectively. We firstly give two notations. For any  $\pi\in\mathfrak{A}^\pi_0$,  $\theta^{\min}(\pi)$ is a PII triplet that satisfies
\[
   g^{\theta^{\min}(\pi)_t}(\pi_t)= \min_{\theta_t \in \Theta_t} g^{\theta_t}(\pi_t),~\forall t\in[0,T].
\]
For any $\theta\in\Theta$, $\pi^{\max}(\theta)$ is an investment policy that satisfies
\[
g^{\theta_t}( \pi^{\max}(\theta)_t ) =  \max_{\pi_t \in \mathcal{O}_t} g^{\theta_t}(\pi_t),~\forall t\in[0,T].
\]
They both exist according to Lemmas \ref{existence of theta} and \ref{existence of pi}.
By the property of saddle point,
\begin{equation}\label{sadddle g}
\min_{\theta\in\Theta}g^{\theta_t}(\pi^{\max}(\theta)_t) = \max_{\pi\in\mathfrak{A}^\pi_0}g^{\theta^{\min}(\pi)_t}(\pi_t) =   g^{\theta^*_t}(\pi^*_t),~\forall t\in[0,T].
\end{equation}

\vskip 10pt
\textit{Step 1:} Calculating $\min\limits_{\theta \in \Theta}\sup\limits_{\pi\in \mathfrak{A}^\pi_0}\sup\limits_{c\in \mathfrak{A}^c_0} G(\pi,c,\theta)$.

It is easy to verify that the $G(\pi,c^*,\theta)$ in Eq. \eqref{optm G with c*} is increasing in the kernel order. Thus, following equalities hold by using Theorem \ref{optimal c} and Eq. \eqref{sadddle g}.
\begin{equation}\nonumber
\!\!\begin{aligned}
&\min_{\theta \in \Theta}\sup_{\pi\in \mathfrak{A}^\pi_0}\sup_{c\in \mathfrak{A}^c_0}G(\pi,c,\theta) \\
=&\min_{\theta \in \Theta} \frac{1}{p}\!\left\{\!\exp\!\left(\int_0^T\!\!\!\frac{p g^{\theta_u}(\pi^{\max}(\theta)_u)}{1-p}du\!\right)\!\!+\!\! \int_0^T\!\!\!\! \exp\!\left(\!\int_0^s  \frac{p g^{\theta_u}(\pi^{\max}(\theta)_u)}{1-p}du\!\right)ds\!\right\}^{\!1\!-\!p}\!\!-\!\!\frac{1}{p},\\
=&\frac{1}{p}\!\left\{\!\exp\!\left(\int_0^T\!\!\!\frac{p g^{\theta^*_u}(\pi^*_u)}{1-p}du\!\right)\!\!+\!\! \int_0^T\!\! \exp\!\left(\!\int_0^s  \frac{p g^{\theta^*_u}(\pi^*_u)}{1-p}du\!\right)ds\!\right\}^{\!1\!-\!p}\!\!-\!\!\frac{1}{p}, \qquad\text{for } p\neq0,
\end{aligned}
\end{equation}
and
\begin{equation}\nonumber
\!\!\begin{aligned}
&\min_{\theta \in \Theta}\sup_{\pi\in \mathfrak{A}^\pi_0}\sup_{c\in \mathfrak{A}^c_0}G(\pi,c,\theta) \\
=&\min_{\theta \in \Theta}\!\!\int_0^T \!\bigg(\!\! \int_0^t g^{\theta_s}(\pi^{\max}(\theta)_s)  ds \!+\!\!g^{\theta_t}(\pi^{\max}(\theta)_t)\bigg)dt-(T+1)\log(T+1),\\
=&\!\!\int_0^T \!\bigg(\!\! \int_0^t g^{\theta^*_s}(\pi^*_s)  ds \!+\!\!g^{\theta^*_t}(\pi^*_t)\bigg)dt-(T+1)\log(T+1),\qquad\text{for } p=0.
\end{aligned}
\end{equation}

\vskip 5pt

\textit{Step 2:} Calculating $\sup\limits_{c\in \mathfrak{A}^c_0}\sup\limits_{\pi\in \mathfrak{A}^\pi_0} \min\limits_{\theta \in \Theta}G(\pi,c,\theta)$.

The calculation for $p\leq0$ is straightforward as in Step 1, while the case of $p>0$ is complicated.
\vskip 5pt
\textit{Case 1: $p\leq0$.}
By Theorem \ref{g_monotony}, Eq. \eqref{sadddle g}, and Theorem \ref{optimal c},
\begin{equation}\nonumber
\!\!\begin{aligned}
&\sup_{c\in \mathfrak{A}^c_0}\sup_{\pi\in \mathfrak{A}^\pi_0} \min_{\theta \in \Theta}G(\pi,c,\theta)
=\sup_{c\in \mathfrak{A}^c_0}\sup_{\pi\in \mathfrak{A}^\pi_0} G(\pi,c,\theta^{\min}(\pi))\\
=&\sup_{c\in \mathfrak{A}^c_0} G(\pi^*,c,\theta^*)
= G(\pi^*,c^*,\theta^*).
\end{aligned}
\end{equation}

\textit{Case 2: $p>0$.}
Firstly, we assert the value of $\sup\limits_{c\in \mathfrak{A}^c_0}\sup\limits_{\pi\in \mathfrak{A}^\pi_0} \min\limits_{\theta \in \Theta}G(\pi,c,\theta)$ does not change if we restrict the sets $\mathfrak{A}^c_0\times \mathfrak{A}^\pi_0\times \Theta$ as
\begin{equation}\label{rest of control}
\begin{aligned}
\{(\pi,c,\theta)\in \mathfrak{A}^c_0\times \mathfrak{A}^\pi_0\times \Theta: c_t\in[0,1],g^{\theta_t}(\pi_t)\geq0, \forall t\in[0,T] \}.
\end{aligned}
\end{equation}
For any $\pi\in \mathfrak{A}^\pi_0$ and $\theta \in \Theta$ satisfying
\[
\{t\in[0,T]:g^{\theta_t}(\pi_t)<0\} \neq \varnothing,
\]
define
\begin{equation}\nonumber
\check{\pi}_t \triangleq\left\{
\begin{aligned}
& 0, &&g^{\theta_t}(\pi_t)<0,\\
& \pi_t , &&g^{\theta_t}(\pi_t)\geq0,
\end{aligned}
\right.
\end{equation}
then $\check{\pi}\in\mathfrak{A}^\pi_0$ and $(\check{\pi}, \theta)\succcurlyeq(\pi, \theta)$. By Theorem \ref{optimal c},
\begin{equation}\nonumber
\begin{aligned}
&\max_{c\in \mathfrak{A}^c_0}G(\pi,c,\theta)\\
=& \frac{1}{p}\!\left\{\!\exp\!\left(\int_0^T\!\!\!\frac{p g^{\theta_u}(\pi_u)}{1-p}du\!\right)\!\!+\!\! \int_0^T\!\! \exp\!\left(\!\int_0^s  \frac{p g^{\theta_u}(\pi_u)}{1-p}du\!\right)ds\!\right\}^{\!1\!-\!p}\!\!-\!\!\frac{1}{p}\\
\leq& \frac{1}{p}\!\left\{\!\exp\!\left(\int_0^T\!\!\!\frac{p g^{\theta_u}(\check{\pi}_u)}{1-p}du\!\right)\!\!+\!\! \int_0^T\!\! \exp\!\left(\!\int_0^s  \frac{p g^{\theta_u}(\check{\pi}_u)}{1-p}du\!\right)ds\!\right\}^{\!1\!-\!p}\!\!-\!\!\frac{1}{p} \\
=&\max_{c\in \mathfrak{A}^c_0}G(\check{\pi},c,\theta),
\end{aligned}
\end{equation}
because $\max\limits_{c\in \mathfrak{A}^c_0}G(\pi,c,\theta)$ is increasing in the kernel order.
Therefore, we can restrict $\mathfrak{A}^c_0\times \mathfrak{A}^\pi_0\times \Theta$  as
\[
\{(\pi,c,\theta)\in \mathfrak{A}_0\times \Theta: g^{\theta_t}(\pi_t)\geq0, \forall t\in[0,T] \}.
\]
According to $g^{\theta_t}(\pi_t)\geq0$, Theorem \ref{optimal c}, and \eqref{est_g_orig} for negative $p$, the optimal $c^*$ must take value in $[0,1]$, thus we can further restrict the control set as \eqref{rest of control}.
Finally, we calculate $\sup\limits_{c\in \mathfrak{A}^c_0}\sup\limits_{\pi\in \mathfrak{A}^\pi_0} \min\limits_{\theta \in \Theta}G(\pi,c,\theta)$ on the restricted control set \eqref{rest of control}.
\begin{equation}\nonumber
\qquad\qquad\begin{aligned}
&\sup_{c\in \mathfrak{A}^c_0}\sup_{\pi\in \mathfrak{A}^\pi_0} \min_{\theta \in \Theta}G(\pi,c,\theta)
=\sup_{c\in \mathfrak{A}^c_0}\sup_{\pi\in \mathfrak{A}^\pi_0} G(\pi,c,\theta^{\min}(\pi))&\\
=&\sup_{c\in \mathfrak{A}^c_0}G(\pi^*,c,\theta^*)
=G(\pi^*,c^*,\theta^*),
\end{aligned}
\end{equation}
by Theorem \ref{optimal c}, Theorem \ref{g_monotony}, and Eq. \eqref{sadddle g}.
$\hfill\square$

\vskip 15pt
\section{Proof of Theorem \ref{measurable selector} }\label{app measurable selector}
Recalling the integral in the expression of the local kernel (cf. Definition \ref{def g CRRA}), we notice its integrand contains $U(1\!+\! x^T z)$, whose value or derivative approaches infinite when $x^T z$ approaches $-1$. Thus, to avoid the singularity of local kernel near the boundary of $\mathcal{O}_t$, we consider the following closed subsets of $\mathcal{O}_t$ for any $t\in[0,T]$ and $n\in\mathds{N}$:
\[
\mathcal{O}_t^n  \triangleq \left\{x\in\R^d : x^Tz\geq-1+\frac{1}{n}, \forall z\in \mathbf{S}_t  \right\},
\]
thus $\mathcal{O}_t = \cup_{n=1}^\infty \mathcal{O}_t^n$.
Since $\mathcal{O}_t$ is bounded by $\kappa_t$, $ \mathcal{O}^n_t$ is bounded too. We only write the proof for $p<0$ here because it is similar and easier for $p\in[0,1)$.

\begin{Lemma}\label{continuous of I}
For any $t\in[0,T]$ and $n\in\mathds{N}$, define
\begin{equation}\nonumber
\mathcal{I}(z,x)\triangleq \left\{
\begin{aligned}
&\frac{p^{-1}((1\!+\!   x^T z)^p-1)\!-\! x^T z}{|z|^{2-\varepsilon} \wedge1}, &&|z|>0,\\
&0, &&z=0,
\end{aligned}
\right.
\end{equation}
then $\mathcal{I}$ is continuous and there exists a constant $C(n,\kappa_t)$ such that
\begin{equation}\label{Lip of I wrt z and x}
\begin{aligned}
&\sup_{ x  \in\mathcal{O}^n_t } \sup_{z\neq \hat{z}\in \mathbf{S}_t }  |\mathcal{I}(z,x) - \mathcal{I}(\hat{z},x)|  \leq C(n,\kappa_t)  |z- \hat{z}|^{\varepsilon\wedge1} ,
\end{aligned}
\end{equation}
\begin{equation}\label{Lip of I wrt x}
\begin{aligned}
&\sup_{z \in \mathbf{S}_t }  \sup_{ x\neq\hat{x} \in\mathcal{O}^n_t }  |\mathcal{I}(z,x) - \mathcal{I}(z,\hat{x})|  \leq C(n,\kappa_t)  |x- \hat{x}|,
\end{aligned}
\end{equation}
and
\begin{equation}\label{bounded of I}
\begin{aligned}
\sup_{(z,x)\in \mathbf{S}_t\times\mathcal{O}^n_t } |\mathcal{I}(z,x)|\leq C(n,\kappa_t).
\end{aligned}
\end{equation}
\end{Lemma}
\begin{proof}
We assume $\mathbf{S}_t$ is a closed and convex set containing the origin, because we can substitute $\mathrm{Conv}(\mathbf{S}_t\cup{0})$ for $\mathbf{S}_t$ without changing $\mathcal{O}^n_t$.
The continuity on $(\mathbf{S}_t\setminus\{0\} )\times \mathcal{O}^n_t$ is obvious, so we only need to verify
\[
\lim_{|z|\rightarrow 0 }\mathcal{I}(z,x)=0,~\forall x\in \mathcal{O}^n_t,
\]
which is ensured by the following two results:
\begin{equation}\nonumber
\begin{aligned}
&~p^{-1}((1\!+\!   x^T z)^p-1)\!-\! x^T z~\leq~0
\Rightarrow \limsup_{|z|\rightarrow 0}\mathcal{I}(z,x)\leq~0,\\
&\liminf_{|z|\rightarrow 0}\mathcal{I}(z,x)\geq \liminf_{|z|\rightarrow 0} \mathcal{I}(\frac{-x}{|x|}|z|,x) = \liminf_{|z|\rightarrow 0} \frac{p^{-1}((1\!+\! |x| |z|)^p\!-\!1)\!-\! |x| |z|}{|z|^{2-\varepsilon} \wedge1}= 0.
\end{aligned}
\end{equation}
\vskip 5pt
Thus, $\mathcal{I}$ is continuous and bounded by the compactness of $\mathbf{S}_t\times\mathcal{O}^n_t$, which is equivalent to \eqref{bounded of I}. Moreover, we can show $(1+x^Tz)^{p-1}-1$ is continuous and bounded too. For convenience, we assume they are bounded by a common constant $C(n,\kappa_t)$ depending on $n$ and $\kappa_t$.

For verifying the $(\varepsilon\wedge1)$-H\"{o}lder continuity of $\mathcal{I}$ with respect to $z$, we firstly study its property on a neighborhood of origin:
\[
\mathcal{N}_\delta = \{ (z,x)\in \mathbf{S}_t\times\mathcal{O}^n_t : |z|\leq \delta \}
\]
for $\delta=(2\kappa_t)^{-1}$.
By Taylor's theorem with Lagrange remainder, for any $(z,x)\in \mathcal{N}_\delta$, we have
\begin{equation}\nonumber
\begin{aligned}
&p^{-1} (1\!+\!   x^T z)^p-1)\!-\! x^T z = \frac{p-1}{2}(x^T z)^2 + \frac{(p-1)(p-2)}{3!}(1+ \xi_\delta)^{p-3} ( x^T z)^3\\
\Rightarrow& \mathcal{I}(z,x) = \frac{p-1}{2}  |x|^2|z|^\varepsilon \cos(\vartheta)^2   + \frac{(p-1)(p-2)}{3!}(1+ \xi_\delta)^{p-3}|x|^3|z|^{(1+\varepsilon)}\cos(\vartheta)^3,
\end{aligned}
\end{equation}
where $\cos(\vartheta) = \frac{x^T z}{|x||z|}$ is the cosine of the angle between $x$ and $z$, and $\xi_\delta$ is a real number between $0$ and $x^T z$. Since $|x|^2  \cos(\vartheta)^2   \leq \kappa_t^2$, the first term is $(\varepsilon\wedge1)$-H\"{o}lder continuous with respect to $z$. Since
\[
|\xi_\delta|\leq \delta\kappa_t = 1/2\Rightarrow   |(1+ \xi_\delta)^{p-3}||x|^3 |\cos(\vartheta)^3|\leq (1/2)^{p-3}\kappa_t^3,
\]
the second term is Lipschitz continuous with respect to $z$. Thus, \eqref{Lip of I wrt z and x} holds for some $C(n,\kappa_t)$  over $\mathcal{N}_\delta$.

Secondly, we study the property on $\{ (z,x)\in \mathbf{S}_t\times\mathcal{O}^n_t : \delta \leq  |z|\leq 1 \}$ and $\{ (z,x)\in \mathbf{S}_t\times\mathcal{O}^n_t : 1 \leq  |z| \}$. Because
\begin{equation}\nonumber
\begin{aligned}
\bigg|\frac{\partial \mathcal{I}}{\partial z_i}\bigg|
&\leq    \frac{|(1+x^Tz)^{p-1}-1|}{|z|^{2-\varepsilon}}|x_i| +   |\mathcal{I}(z,x)|  \frac{(2-\varepsilon)|z_i|}{|z|^2}  \\
&\leq C(n,\kappa_t)[ (2\kappa_t)^{2-\varepsilon}\kappa_t + (2-\varepsilon)2\kappa_t], ~\forall \delta \leq  |z|\leq 1,
\end{aligned}
\end{equation}
and
\begin{equation}\nonumber
\begin{aligned}
\bigg|\frac{\partial \mathcal{I}}{\partial z_i}\bigg|
&\leq     |(1+x^Tz)^{p-1}-1| |x_i|   \leq C(n,\kappa_t) \kappa_t, ~\forall 1\leq |z|\leq \kappa_t,
\end{aligned}
\end{equation}
$\mathcal{I}$ is differentiable with bounded derivative on above tow domains. Thus $\mathcal{I}$ is Lipschitz continuous with respect to $z$ over these two compact domains. Recalling the result on domain $\mathcal{N}_\delta$, we conclude that $\mathcal{I}$ is $(\varepsilon\wedge1)$-H\"{o}lder continuous with respect to $z$ over $\mathbf{S}_t\times\mathcal{O}^n_t$ and the H\"{o}lder constant is independent of $x$, which is equivalent to \eqref{Lip of I wrt z and x} for a large enough $ C(n,\kappa_t) $.
Similarly, the Lipschitz continuity of  $\mathcal{I}$ with respect to $x$ can be proved, which is omitted here.
\end{proof}

\begin{Lemma}\label{continuity of theta}
For any $t\in[0,T]$ and $n\in\mathds{N}$, the map
\[
(y,M,\mu)\mapsto \sup_{x\in\mathcal{O}^n_t} g_t^{(y,M,\mu)}(x)
\]
is continuous on $\Theta_t$.
\end{Lemma}
\begin{proof}
For any $(y, M, \mu), (\hat{y}, \hat{M}, \hat{\mu})\in \Theta_t$ and $x\in\mathcal{O}^n_t$,
\begin{equation}\nonumber
\begin{aligned}
&\bigg|x^Ty \!-\! \frac{1\!-\!p}{2}x^T M x - x^T\hat{y} \!+\! \frac{1\!-\!p}{2}x^T \hat{M} x\bigg|\\
\leq& |x^T(y -\hat{y})| + \frac{1\!-\!p}{2}|x^T (M-\hat{M}) x|\\
\leq& \big(\kappa_t + \frac{1\!-\!p}{2}\kappa_t^{2}\big) ( d_2( y, \hat{y})\vee d_2( M, \hat{M})  ).
\end{aligned}
\end{equation}
By Lemma \ref{continuous of I}, we have
\begin{equation}\nonumber
\begin{aligned}
&\bigg|\!\int_{\R^d}\!\!\! \left( \frac{(1\!\!+\!\! x^T z\!)^p\!\!-\!\!  1}{p}\!-\! x^T z\!\right) \!\mu(dz)
- \int_{\R^d}\!\!\! \left( \frac{(1\!\!+\!\! x^T z\!)^p\!\!-\!\!  1}{p}\!-\! x^T z\!\right) \!\hat{\mu}(dz)\bigg|\\
=&\bigg|\int_{ \mathbf{S}_t}\!\! \mathcal{I}(z,x) (|z|^{2-\varepsilon}\wedge1) \mu(dz) -\int_{ \mathbf{S}_t}\!\! \mathcal{I}(z,x) (|z|^{2-\varepsilon}\wedge1) \hat{\mu}(dz)\bigg|\\
\leq& C(n,\kappa_t)  d_{BH}^{~\varepsilon}( |z|^{2-\varepsilon}\wedge1.\mu, |z|^{2-\varepsilon}\wedge1.\hat{\mu})\\
\leq& C(n,\kappa_t)    d^{\varepsilon}_{\mathcal{L}}(\mu,\hat{\mu}).
\end{aligned}
\end{equation}
Therefore, for fixed $t$ and $n$, the family of functions
\[
\{(y,M,\mu)\mapsto g_t^{(y,M,\mu)}(x)\}_{x\in\mathcal{O}^n_t}
\]
is equicontinuous and uniformly Lipschitz continuous with Lipschitz constant
\[
\kappa_t + \frac{1\!-\!p}{2}\kappa_t^{ 2}+C(n,\kappa_t) ,
\]
under the maximum metric $d_\mathcal{C}^{\varepsilon}$.
Thus $\sup\limits_{x\in\mathcal{O}^n_t} g_t^{\cdot}(x)$ is continuous on $\Theta_t$.
\end{proof}

The proofs of the following two lemmas use Kuratowski-Ryll-Nardzewski measurable selection theorem and measurable maximum theorem many times, so we state them here for convenience. The details can be found in Chapter 18, Aliprantis and Border (2006).

\begin{Def}[Carath\'{e}odory function]
Let $(S, \Sigma)$ be a measurable space, and let $X$ and $Y$ be topological
spaces. A function $f: S \times X \rightarrow Y$ is a Carath\'{e}odory function if:
\begin{enumerate}[(1)]
	\item for each $x\in X$, the function $f^x = f(\cdot,x ): S\rightarrow Y$ is $(\Sigma, \mathcal{B}_Y)$-measurable;
    and
	\item for each $s\in S$, the function $f^s = f(s,\cdot ): X\rightarrow Y$ is continuous.
\end{enumerate}
\end{Def}

\begin{theorem}[Kuratowski-Ryll-Nardzewski measurable selection theorem]
A weakly measurable correspondence with nonempty closed values from a measurable space into
a Polish space admits a measurable selector.
\end{theorem}

\begin{theorem}[Measurable maximum theorem]
Let $X$ be a separable metrizable space and $(S, \Sigma)$ a measurable space. Let $\varphi:S\twoheadrightarrow X$ be a weakly measurable correspondence with nonempty compact values, and suppose $f: S \times X \rightarrow \R$ is a Carath\'{e}odory function. Define the value function $m: S \rightarrow \R$ by $m(s) =\max\limits_{x\in\varphi(s)}f(s,x)$, and the correspondence $\mu: S\twoheadrightarrow X$ of maximizers by
\[
\mu(s) = \{ x\in \varphi(s): f(s,x) = m(s)\}.
\]
Then:
\begin{enumerate}[(1)]
	\item The value function $m$ is measurable.
	\item The argmax correspondence  $\mu$ has nonempty and compact values.
    \item The argmax correspondence $\mu$ is measurable and admits a measurable selector.
\end{enumerate}
\end{theorem}

\begin{Lemma}\label{existence of theta}
There exists a $\theta^*\in\Theta$ such that $\theta^*_t$ attains the minimum of
$\sup_{x\in\mathcal{O}_t} g_t^{\cdot}(x)$ on $\Theta_t$ for every $t\in[0,T]$.
\end{Lemma}

\begin{proof}
By Lemma \ref{continuity of theta}, the map
$
(t, (y,M,\mu)) \mapsto \max\limits_{x\in\mathcal{O}^n_t} g_t^{(y,M,\mu)}(x)
$
is a Carath\'{e}odory function. Thus, for any $n\in\mathds{N}$,
$
t \mapsto \min_{\Theta_t} \max_{\mathcal{O}^n_t} g_t
$
is measurable by the measurable maximum theorem. Referring to lemma 3.2 of Neufeld and Nutz (2018), we know
\begin{equation}\nonumber
\begin{aligned}
 \min_{\Theta_t} \sup_{\mathcal{O}_t} g_t  = \lim_{n\rightarrow\infty} \min_{\Theta_t} \max_{\mathcal{O}^n_t} g_t
\Rightarrow   t \mapsto \min_{\Theta_t} \sup_{\mathcal{O}_t} g_t \text{ is measurable.}
\end{aligned}
\end{equation}
Therefore, the map
\[
(t, (y,M,\mu)) \mapsto \bigg(\max_{x\in\mathcal{O}^n_t} g_t^{(y,M,\mu)}(x) -  \min_{\Theta_t} \sup_{\mathcal{O}_t} g_t\bigg)
\]
is a Carath\'{e}odory function and the correspondence
\[
t \twoheadrightarrow  \bigg\{  (y,M,\mu)\in \Theta_t:  \max_{x\in\mathcal{O}^n_t} g_t^{(y,M,\mu)}(x) \leq \min_{\Theta_t} \sup_{\mathcal{O}_t} g_t\bigg\}
\]
is measurable (cf. corollary 18.8, Aliprantis and Border (2006)). Since \\
$\{\max_{\mathcal{O}^n_t} g_t^{(y,M,\mu)}\}_{n\in\mathds{N}}$ increasingly converges to $\sup_{\mathcal{O}_t} g_t^{(y,M,\mu)}$, we have
\begin{equation}\nonumber
\begin{aligned}
\bigg\{ (y,M,\mu)\in \Theta_t:  \sup_{x\in\mathcal{O}_t} g_t^{(y,M,\mu)}(x) \leq \min_{\Theta_t} \sup_{\mathcal{O}_t} g_t\bigg\}\\
=\bigcap_{n=1}^\infty \bigg\{ (y,M,\mu)\in \Theta_t:  \max_{x\in\mathcal{O}_t^n} g_t^{(y,M,\mu)}(x) \leq \min_{\Theta_t} \sup_{\mathcal{O}_t} g_t\bigg\},
\end{aligned}
\end{equation}
thus the correspondence
\[
t \twoheadrightarrow  \bigg\{  (y,M,\mu)\in \Theta_t:  \sup_{x\in\mathcal{O}_t} g_t^{(y,M,\mu)}(x) \leq \min_{\Theta_t} \sup_{\mathcal{O}_t} g_t\bigg\}
\]
is measurable (cf. lemma 18.4, Aliprantis and Border (2006)). Finally, by Kuratowski-Ryll-Nardzewski measurable selection theorem, there exists a measurable function $\theta^*$ such that $\theta^*_t \in \Theta_t$ and
\[
\sup_{x\in\mathcal{O}_t} g_t^{\theta^*_t}(x) \leq \min_{\Theta_t} \sup_{\mathcal{O}_t} g_t, ~\forall t\in[0,T].
\]

\end{proof}

\begin{Lemma}\label{existence of pi}
There exists a $\pi^*\in\mathfrak{A}^\pi_0$ such that $\pi^*_t$ attains the maximum of
$\min\limits_{(y, M, \mu)\in \Theta_t}g_t^{(y, M, \mu)}(\cdot)$ on $\mathcal{O}_t$ for every $t\in[0,T]$.
\end{Lemma}

\begin{proof}
Using \eqref{Lip of I wrt x} in Lemma \ref{continuous of I} and imitating the proof of Lemma \ref{continuity of theta}, we can prove that the map
\[
x \mapsto \min_{(y,M,\mu)\in \Theta_t} g_t^{(y,M,\mu)}(x)
\]
is continuous on $\mathcal{O}^n_t$ for any $t\in[0,T]$ and $n\in\mathds{N}$.
Similarly to the proof of Lemma \ref{existence of theta}, the correspondence
\[
t \twoheadrightarrow  \bigg\{  x\in \mathcal{O}^n_t:  \min_{(y,M,\mu)\in \Theta_t} g_t^{(y,M,\mu)}(x) \geq \sup_{\mathcal{O}_t}\min_{\Theta_t}  g_t\bigg\}
\]
is measurable for any $n\in\mathds{N}$, thus
\[
t \twoheadrightarrow  \bigg\{  x\in \mathcal{O}_t:  \min_{(y,M,\mu)\in \Theta_t} g_t^{(y,M,\mu)}(x) \geq \sup_{\mathcal{O}_t}\min_{\Theta_t}  g_t\bigg\}
\]
is measurable as an union correspondence. It can be verified that this correspondence is nonempty by the assumption that $\mathbf{S}_t $ is closed. Indeed, for any $\check{x}\in \partial\mathcal{O}_t$, there exists a $(\check{y},\check{M},\check{\mu})\in \Theta_t$ such that
\[
\exists\check{z}\in \mathrm{supp}(\check{\mu}), x^Tz=-1.
\]
Thus, as $x$ approaches $\check{x}$, $ g_t^{(\check{y},\check{M},\check{\mu})}(x)$ decreases at infinite rate so that the supremum of $\min_{\Theta_t}g_t$ over $\mathcal{O}_t$ can be attained.
Finally, by Kuratowski-Ryll-Nardzewski measurable selection theorem, a measurable $\pi^*$ exists satisfying
\[
\min_{(y,M,\mu)\in \Theta_t} g_t^{(y,M,\mu)}(\pi^*_t) \geq \sup_{\mathcal{O}_t}\min_{\Theta_t}  g_t, ~\forall t\in[0,T].
\]

\end{proof}
\vskip 10pt
\section{Proofs of Theorems \ref{optimal c} and \ref{optimal c CARA}}\label{app opt c}
We firstly introduce the Lemma \ref{verify thm}, which is the so-called verification theorem in dynamic programming.
\begin{Lemma}\label{verify thm}
For any $p\in (-\infty,0)\cup(0,1)$, $\theta \in \Theta$, and $\pi\in \mathfrak{A}^\pi_0$, there exists an unique function $V\in C^1([0,T])$ satisfying the following HJB equation:
\begin{equation}\label{HJB eq}
\left\{
\begin{aligned}
&\frac{dV(t)}{dt}+ g^{\theta_t}(\pi_t)(1+pV(t))+\sup_{c_t\geq 0}\{U(c_t)-c_t ( 1+pV(t))\}=0,\\
&V(T)= 0.
\end{aligned}
\right.
\end{equation}
The range of $V$ is $[V_{\min}, V_{\max}]$, where
\begin{equation}\nonumber
\begin{aligned}
&g_{\max} \triangleq \sup_{t\in[0,T]}g^{\theta_t}(\pi_t),
&g_{\min} \triangleq \inf_{t\in[0,T]}g^{\theta_t}(\pi_t),
\end{aligned}
\end{equation}
and
\begin{equation}\label{min of V}
V_{\min} \triangleq\left\{
\begin{aligned}
&0,&&g_{\min}\geq\frac{1-p}{-p},\\
&-\infty,&&p<0,g_{\min}\leq0,\\
&\frac{1}{p}\bigg[\bigg(\frac{-pg_{\min}}{1-p}\bigg)^{p-1}-1\bigg],&&\mathrm{otherwise},\\
\end{aligned}
\right.
\end{equation}
and
\begin{equation}\label{max of V}
V_{\max} \triangleq\left\{
\begin{aligned}
&0,&&g_{\max}\leq\frac{1-p}{-p},\\
&+\infty,&&p>0,g_{\max}\geq0,\\
&\frac{1}{p}\bigg[\bigg(\frac{-pg_{\max}}{1-p}\bigg)^{p-1}-1\bigg],&&\mathrm{otherwise}.\\
\end{aligned}
\right.
\end{equation}
Especially,  $V(0)$ is the upper bound of $G(\pi,\cdot,\theta)$ over $\mathfrak{A}^c_0$, i.e.,
$$ \sup_{c\in \mathfrak{A}^c_0}G(\pi,c,\theta)\leq V(0).$$
\end{Lemma}
\begin{proof}
There are three parts in this proof. The first part proves that any $C^1$ solution of HJB equation \eqref{HJB eq} takes value in $[V_{\min}, V_{\max}]$. The second part shows the existence and uniqueness by Picard-Lindel\"{o}f theorem. The third part verifies that $V(0)$ is the upper bound.

\textit{Boundary.}
If $V$ is a solution of HJB equation in $C^1([0,T])$, we assert
\[
1+pV(t)>0, ~\forall t\in[0,T].
\]
If not, there exists $t_0$ such that $1+pV(t_0)\leq0$, thus $\frac{dV(t)}{dt}\big|_{t=t_0} = -\infty$. This contradicts to the continuity of $V$ on $[0,T]$. Therefore, $c_t = ( 1+pV(t))^{\frac{1}{p-1}}\geq0$ attains the supremum of
$U(c_t)-c_t ( 1+pV(t))$ and the HJB equation \eqref{HJB eq} becomes
\begin{equation}\nonumber
\left\{
\begin{aligned}
&\frac{dV(t)}{dt}= f(t,V(t)),~\forall t\in[0,T],\\
&V(T)= 0,
\end{aligned}
\right.
\end{equation}
where $f(t,x) \triangleq -g^{\theta_t}(\pi_t)(1+px)-\frac{1-p}{p}( 1+px)^{\frac{p}{p-1}}$. By the knowledge of dynamical system, it is sufficient to verify the following two inequalities for demonstrating that $[V_{\min}, V_{\max}]$ is the range of $V$:
\[
\frac{dV(t)}{dt}\bigg|_{V(t)=V_{\max}}\geq0, ~\text{if }V_{\max}<+\infty,
\]
\[
\frac{dV(t)}{dt}\bigg|_{V(t)=V_{\min}}\leq0,  ~\text{if }V_{\min}>-\infty.
\]
They can be easily verified when $V_{\max}$ and $V_{\min}$ are defined by \eqref{min of V} and \eqref{max of V} respectively.

\textit{Existence and uniqueness.}
By the Picard-Lindel\"{o}f theorem, we only need to show $f$ is Lipschitz continuous on $[V_{\min}, V_{\max}]$. Noticing $\frac{1-p}{p}( 1+px)^{\frac{p}{p-1}}$ is Lipschitz continuous in the domain that avoids the singular point $\frac{1}{-p}$, we can conclude $f$ is Lipschitz continuous because the interval $[V_{\min}, V_{\max}]$, which is defined by \eqref{min of V} and \eqref{max of V}, strictly avoids $\frac{1}{-p}$.

\textit{Upper bound.}
For any $c\in \mathfrak{A}^c_0$, by Newton-Leibniz formula and $V(T)=0$,
\begin{equation}\label{verify ineq}
\begin{aligned}
V(0)
=& V(0)-\exp\bigg(  \!\int_0^T \!p(g^{\theta_u}(\pi_u) \!-\!c_u) du\!\bigg)V(T)\\
=&-\!\!\int_0^T \!\!\!\exp\bigg(  \!\int_0^t \!\!p(g^{\theta_u}(\pi_u) \!-\!c_u) du\!\bigg)(V'(t)+p(g^{\theta_t}(\pi_t) \!-\!c_t)V(t))dt\\
\geq&\int_0^T \!\exp\bigg(  \!\int_0^t \!\!p(g^{\theta_u}(\pi_u) \!-\!c_u) du\!\bigg)(g^{\theta_t}(\pi_t) \!-\!c_t+U(c_t))dt.
\end{aligned}
\end{equation}
Due to the arbitrariness of $c$, we conclude $V(0)\geq \sup\limits_{c\in \mathfrak{A}^c_0}G(\pi,c,\theta)$.
\end{proof}

In above lemma, we do not prove a corresponding conclusion for the logarithmic utility ($p=0$), because the variational method is enough for logarithmic utility. See the following proof.

\vskip 10pt
\textit{\textbf{Proof of Theorem \ref{optimal c}}.}
This proof has two steps. The first step is demonstrating the optimality of $c^*$, and the second step is estimating the range of $c^*$.

\vskip 5pt
\textit{Step 1: Optimality.}
We use different approaches to demonstrate the optimality for logarithmic utility and power utility. The variational method solves the case of $p=0$, while the dynamic programming works for $p\neq0$.

\textit{Case 1: $p=0$.}
For any $\tilde{c}\in \mathfrak{A}^c_0$, let $\bar{c}  \triangleq \tilde{c} - c^*$ be the variation. For any $\delta\in[0,1]$,
\begin{equation}\nonumber
\begin{aligned}
&\frac{d}{d\delta}G(\pi, c^*+\delta \bar{c}, \theta)
\\
=&\frac{d}{d\delta}\int_0^T\!\!\!\bigg(  \int_0^t\!\! -(c^*_s+\delta \bar{c}_s) ds+\log(c^*_t+\delta \bar{c}_t)-(c^*_t+\delta \bar{c}_t)\bigg) dt \\
=&\int_0^T\!\!\!  \bigg(\int_0^t\!\! -\bar{c}_s  ds+  \frac{\bar{c}_t}{c^*_t+\delta \bar{c}_t}-\bar{c}_t \bigg)dt\\
=&\int_0^T\!\!\!    \bigg(\frac{1}{c^*_t+\delta \bar{c}_t}-(T-t+1)\bigg)\bar{c}_t dt.
\end{aligned}
\end{equation}
For any positive or negative $\bar{c}_t$, notice
\begin{equation}\nonumber
\begin{aligned}
c^*_t = (T-t+1)^{-1}, \delta>0
 ~\Rightarrow ~\left(\frac{1}{c^*_t+\delta \bar{c}_t}-(T-t+1)\right)\bar{c}_t \leq 0,
\end{aligned}
\end{equation}
thus $G(\pi, \tilde{c}, \theta)\leq G(\pi, c^*, \theta)$.

\textit{Case 2: $p\neq0$.}
Lemma \ref{verify thm} shows that $V(0)$ is a upper bound of $G(\pi,\cdot,\theta)$ over $\mathfrak{A}^c_0$. So it is sufficient to verify $G(\pi,c^*,\theta) = V(0)$, where $c^*$ is defined by \eqref{exp of c*}. Because $c_t = ( 1+pV(t))^{\frac{1}{p-1}}$ attains the supremum in HJB equation \eqref{HJB eq}, we substitute $V(t)=\frac{1}{p}(c_t^{p-1}-1)$ into \eqref{HJB eq} and get
\begin{equation}\nonumber
\left\{
\begin{aligned}
&\frac{dc_t}{dt}= \frac{pg_t}{1-p} c_t+ c^2_t,\\
&c_T=1.
\end{aligned}
\right.
\end{equation}
This is a Riccati equation and $c^*$ is the unique solution. Since $c^*$ satisfies
$V'(t)+p(g_t \!-\!c^*_t)V(t)=g_t \!-\!c^*_t+U(c^*_t)$,
the inequality  in \eqref{verify ineq} becomes an equality, i.e.,
$$
G(\pi,c^*,\theta) = V(0) =\frac{(c^*_0)^{p-1}-1}{p}.
$$

\vskip 5pt
\textit{Step 2: Boundary.}
The range \eqref{range of c} of $c^*$ could be immediately obtained by \eqref{min of V} and \eqref{max of V}, once noticing $c_t = ( 1+pV(t))^{\frac{1}{p-1}}$ is decreasing with respect to $V(t)$ for $p>0$ and increasing for $p<0$.
$\hfill\square$

\vskip 15pt

\begin{Lemma}\label{verify thm CARA}
For any $\theta \in \Theta$ and $\Pi\in \mathfrak{A}^\Pi_0$, there exists an unique $V\in C^1([0,T])$ satisfying the following HJB equation:
\begin{equation}\label{HJB eq CARA}
\left\{
\begin{aligned}
&\frac{dV(t)}{dt}+ q_th_t^{\theta_t}(\Pi_t)(1\!-\!aV(t))+\sup_{D_t\in\R}\{U(D_t)-q_tD_t ( 1 \!-\! a V(t))\}=0,\\
&V(T)= 0.
\end{aligned}
\right.
\end{equation}
The range of $V$ is $[V_{\min}, V_{\max}]$, where
\begin{equation}\nonumber
\begin{aligned}
&h_{\max} \triangleq \sup_{t\in[0,T]}h_t^{\theta_t}(\Pi_t),
&h_{\min} \triangleq \inf_{t\in[0,T]}h_t^{\theta_t}(\Pi_t),
\end{aligned}
\end{equation}
and
\begin{equation}\label{min of V CARA}
V_{\min}=\left\{
\begin{aligned}
&0,&&h_{\min}\geq\frac{1}{a}(1\!+\!\log(T\!+\!1)),\\
&\frac{1}{a}(1\!-\!(T\!+\!1)e^{1-ah_{\min}}),&&h_{\min}<\frac{1}{a}(1\!+\!\log(T\!+\!1)),
\end{aligned}
\right.
\end{equation}
and
\begin{equation}\label{max of V CARA}
V_{\max}=\left\{
\begin{aligned}
&\frac{1}{a}(1\!-\!e^{1-ah_{\max}}),&&h_{\max}\geq\frac{1}{a},\\
&0,&&h_{\max}\leq\frac{1}{a}.
\end{aligned}
\right.
\end{equation}
Especially, $V(0)$ is the upper bound of global kernel over $\mathfrak{A}_0^D$, i.e.,
$$ \sup_{D\in \mathfrak{A}_0^D}H(\Pi,D,\theta)\leq V(0).$$
\end{Lemma}
\begin{proof}
Same as Lemma \ref{verify thm}, this proof also has three parts.

\textit{Boundary.}
If $V$ is the solution of HJB equation in $C^1([0,T])$, we assert $1-aV(t)>0$ for any $t\in[0,T]$. If not, there exists $t_0$ satisfying $1-aV(t_0)\leq0$, thus $\frac{dV(t)}{dt}\big|_{t=t_0} = -\infty$. This contradicts to the continuity of $V$ on $[0,T]$. Therefore, $D_t = \frac{1}{-a}\log(q_t(1-aV(t)))$ is well-defined and attains the supremum of $U(D_t)-q_tD_t ( 1 - a V(t))$. The HJB equation \eqref{HJB eq CARA} becomes
\begin{equation}\nonumber
\left\{
\begin{aligned}
&\frac{dV(t)}{dt}= f(t,V(t)),~\forall t\in[0,T],\\
&V(T)= 0,
\end{aligned}
\right.
\end{equation}
where $f(t,x) \triangleq \frac{1}{a}q_t (1-ax)[1-a h_t^{\theta_t}(\Pi_t)-\log(q_t(1-ax))]$.
By the knowledge of dynamical system, it is sufficient to verify the following two inequalities for demonstrating that $[V_{\min}, V_{\max}]$ is the range of $V$:
\[
\frac{dV(t)}{dt}\bigg|_{V(t)=V_{\max}}\geq0, ~\text{if }V_{\max}<+\infty,
\]
\[
\frac{dV(t)}{dt}\bigg|_{V(t)=V_{\min}}\leq0,  ~\text{if }V_{\min}>-\infty.
\]
They both hold when $V_{\max}$ and $V_{\min}$ are defined by \eqref{min of V CARA} and \eqref{max of V CARA}.

\textit{Existence and uniqueness.}
By the Picard-Lindel\"{o}f theorem, the conclusion holds because $f$ is  Lipschitz continuous on $[V_{\min}, V_{\max}]$.

\textit{Upper bound.}
For any $D\in \mathfrak{A}_0^D$, by Newton-Leibniz formula and $V(T)=0$,
\begin{equation}\label{verify ineq CARA}
\begin{aligned}
V(0)
=&V(0)\!-\!\exp\bigg( \!\int_0^T \!\!-\!aq_sh_s \!+\!aq_sD_s ds\bigg)V(T)\\
=&\!-\!\!\int_0^T \!\!\!\exp\bigg( \!\int_0^t \!\!\!- aq_sh_s \!+\!aq_sD_s du\!\bigg)\bigg(V'(t)\!-\!aq_t(h_t \!-\!D_t)V(t)\bigg)dt\\
\geq&\int_0^T\!\!\! \exp\bigg( \!\int_0^t \!\!\!- aq_sh_s +aq_sD_s du\!\bigg)\bigg(q_th_t \!-\! q_tD_t + U(D_t)\bigg)dt.
\end{aligned}
\end{equation}
By the arbitrariness of $D$, we conclude
$
V(0)\geq\sup\limits_{D\in \mathfrak{A}_0^D}H(\Pi,D,\theta).
$
\end{proof}

\vskip 10pt
\textit{\textbf{Proof of Theorem \ref{optimal c CARA}}.}
\vskip 5pt
Lemma \ref{verify thm CARA} gives an upper bound of $H(\Pi,\cdot,\theta)$ over $\mathfrak{A}_0^D$. Thus we only need to verify $H(\Pi,D^*,\theta) = V(0)$ where $D^*$ is defined by \eqref{exp of c* CARA}.
Recalling  $D_t = \frac{1}{-a}\log(q_t(1-aV(t)))$ attains the supremum in HJB equation \eqref{HJB eq CARA} and substituting $V(t)=\frac{1}{a}(1-q_t^{-1}e^{-aD_t})$ into \eqref{HJB eq CARA}, we obtain
\begin{equation}\nonumber
\left\{
\begin{aligned}
&\frac{dD_t}{dt}= q_t(D_t-h_t),\\
&D_T=0.
\end{aligned}
\right.
\end{equation}
$D^*$ defined in \eqref{exp of c* CARA}  is exactly the unique solution. Because $D^*$ satisfies
\[
V'(t)-aq_t(h_t \!-\!D^*_t)V(t)= q_th_t - q_tD^*_t + U(D^*_t),
\]
we conclude that $H(\Pi,D^*,\theta) = V(0) =\frac{1}{a}(1-q_0^{-1}e^{-aD_0})$ by \eqref{verify ineq CARA}.
$\hfill\square$

\vskip 10pt
{\bf Acknowledgements.}
The authors acknowledge the support from the National Natural Science Foundation of China (Grant No.11471183, No.11871036). The authors also thank the members of the group of Insurance Economics and Mathematical Finance at the Department of Mathematical Sciences, Tsinghua University for their feedbacks and useful conversations.

\vskip 10pt
\setcounter{equation}{0}

\end{document}